\documentclass[a4paper,11pt]{article}

\usepackage[T1]{fontenc}
\usepackage{lmodern}

\usepackage{dsfont}

\usepackage[latin1]{inputenc}
\usepackage{amsmath}
\usepackage{amsthm}
\usepackage{amssymb}
\usepackage{mathrsfs}
\usepackage{graphicx}
\usepackage[all]{xy}
\usepackage{hyperref}

\usepackage{stmaryrd}
\usepackage{caption}

\usepackage{abstract} 

\newtheorem{thm}{Theorem}[section]

\newtheorem{claim}[thm]{Claim}
\newtheorem{fact}[thm]{Fact}

\newtheorem{lemma}[thm]{Lemma}
\newtheorem{prop}[thm]{Proposition}

\theoremstyle{definition}
\newtheorem{definition}[thm]{Definition}
\newtheorem{ex}[thm]{Example}
\newtheorem{remark}[thm]{Remark}
\newtheorem{question}[thm]{Question}
\newtheorem{problem}[thm]{Problem}

\title{Negative curvature in graph braid groups}
\date{\today}
\author{Anthony Genevois}

\begin{document}

\maketitle

\begin{abstract}
In this article, we initiate a geometric study of graph braid groups. More precisely, by applying the formalism of special colorings introduced in a previous article, we determine precisely when a graph braid group is Gromov-hyperbolic, toral relatively hyperbolic, and acylindrically hyperbolic. 
\end{abstract}

\tableofcontents

\section{Introduction}

Given a topological space $X$ and an integer $n \geq 1$, the \emph{configuration space of $n$ points in $X$} is 
$$C_n(X)= \{ (x_1, \ldots, x_n) \in X^n \mid \text{for every $i \neq j$, $x_i \neq x_j$} \}.$$
So a point in $C_n(X)$ is the data of $n$ ordered and pairwise distinct points in $X$. The corresponding configuration space of unordered collections of points is the quotient $UC_n(X)= C_n(X) / \mathfrak{S}_n$ where the symmetric group $\mathfrak{S}_n$ acts by permuting the coordinates. Given an initial configuration $\ast \in UC_n(X)$, the \emph{braid group} $B_n(X,\ast)$ is the fundamental group of $UC_n(X)$ based at $\ast$. Basically, it is the group of trajectories of $n$ points in $X$, starting and ending at $\ast$ (not necessarily in the same order), up to isotopy. Most of the time, if $X$ is ``sufficiently connected'', the braid group does not depend on the basepoint $\ast$ (up to isomorphism), and by abuse of notation we denote by $B_n(X)$ the braid group. 

The most famous braid groups are the braid groups over a disk, introduced and studied by Artin. We refer to the survey \cite{BraidsSurvey} and references therein for more information on these groups and their links with other areas of mathematics. It is worth noticing that braid groups over $n$-dimensional manifolds are trivial when $n \geq 3$. Consequently, it is natural to focus on one- and two-dimensional spaces, justifying the interest in \emph{graph braid groups} and \emph{surface braid groups}. In this article, we are interested in graph braid groups and we refer to the survey \cite{SurfaceBraidsSurvey} for more information on surface braid groups. 

Graph braid groups seem to have been introduced for the first time in group theory in \cite{GraphBraidGroups}. (Although topology of configuration spaces on graphs were studied before.) So far, the main problems which have been considered are: the links between graph braid groups and right-angled Artin groups \cite{CrispWiest, ConnollyDoig, KimKoPark, EmbeddingRAAGgraphbraidgroup, FarleySabalka}; computing presentations of graph braid groups  \cite{PresentationsGraphBraidGroups, Kurlin}; (co)homological properties of graph braid groups \cite{KoPark, KimKoPark, FarleySabalka}. However, their geometry remains essentially unknown. Our goal is to initiate a geometric study of graph braid groups by applying the formalism introduced in \cite{SpecialRH} in order to investigate their negatively-curved properties. More precisely, we are interested in Gromov-hyperbolicity, relative hyperbolicity (see \cite{OsinRelativeHyp, MR2684983} for more information), and acylindrical hyperbolicity (see \cite{OsinAcyl} for more information).

First of all, we are able to determine precisely when a graph braid group is hyperbolic (as defined by Gromov). More precisely, we obtain the following characterisation:

\begin{thm}
Let $\Gamma$ be a compact and connected one-dimensional CW-complex. 
\begin{itemize}
	\item The braid group $B_2(\Gamma)$ is hyperbolic if and only if $\Gamma$ does not contain a pair of disjoint induced cycles.
	\item The braid group $B_3(\Gamma)$ is hyperbolic if and only if $\Gamma$ is a tree, or a sun graph, or a rose graph, or a pulsar graph.
	\item For every $n \geq 4$, the braid group $B_n(\Gamma)$ is hyperbolic if and only if $\Gamma$ is a rose graph.
\end{itemize}
\end{thm}

Interestingly, the graph braid group of a rose graph turns out to be free, so that the (cohomologic or asymptotic) dimension of a hyperbolic graph braid group must be at most three. Finding non-free hyperbolic graph braid groups is an interesting problem; see Problem \ref{problem:hypbraidgroups} and its related discussion. 

Next, we are able to show that essentially all braid groups are acylindrically hyperbolic. Roughly speaking, a graph braid group (which is not cyclic) is acylindrically hyperbolic if and only if the corresponding graph is connected. For the implication, see Lemma \ref{lem:product} and the observation which follows; for the converse, we prove:

\begin{thm}
Let $\Gamma$ be a connected compact one-dimensional CW-complex and $n \geq 2$ an integer. The braid group $B_n (\Gamma)$ is either cyclic or acylindrically hyperbolic. 
\end{thm}

Finally, we turn to the relative hyperbolicity of graph braid groups. Our main result in this direction is the characterisation of graph braid groups which are toral relatively hyperbolic, i.e., hyperbolic relative to free abelian subgroups.

\begin{thm}
Let $\Gamma$ be a connected compact one-dimensional CW-complex. 
\begin{itemize}
	\item The braid group $B_2(\Gamma)$ is hyperbolic relative to abelian subgroups if and only if $\Gamma$ does not contain an induced cycle which is disjoint from two other induced cycles.
	\item The braid group $B_3(\Gamma)$ is hyperbolic relative to abelian subgroups if and only if $\Gamma$ does not contain an induced cycle disjoint from two other induced cycles; nor a vertex of degree at least four disjoint from an induced cycle; nor a segment between two vertices of degree three which is disjoint from an induced cycle; nor a vertex of degree three which is disjoint from two induced cycles; nor two disjoint induced cycles one of those containing a vertex of degree three.
	\item The braid group $B_4(\Gamma)$ is hyperbolic relative to abelian subgroups if and only if $\Gamma$ is a rose graph, or a segment linking two vertices of degree three, or a cycle containing two vertices of degree three, or two cycles glued along a non-trivial segment.
	\item For every $n \geq 5$, the braid group $B_n(\Gamma)$ is hyperbolic relative to abelian subgroups if and only if $\Gamma$ is a rose graph. If so, $B_n(\Gamma)$ is a free group. 
\end{itemize}
\end{thm}

Interestingly, it follows that the (asymptotic or cohomologic) dimension of a graph braid group which is hyperbolic relative to abelian subgroups must be at most four; in particular, such a group cannot contain $\mathbb{Z}^5$ as a subgroup. 

The problem of charaterising relatively hyperbolic graph braid groups seems to be much more difficult in full generality, and we were not able to solve it. Nevertheless, we prove the following sufficient criterion for braid groups on graphs with two particles. (For a description of the associated peripheral subgroups, see the more precise statement of Theorem \ref{thm:braidgrouprh} in Section \ref{section:RH}.)

\begin{thm}
Let $\Gamma$ be a connected compact one-dimensional CW-complex, and let $\mathcal{G}$ be a collection of subgraphs of $\Gamma$ satisfying the following conditions:
\begin{itemize}
	\item every pair of disjoint simple cycles of $\Gamma$ is contained in some $\Lambda \in \mathcal{G}$;
	\item for every distinct $\Lambda_1,\Lambda_2 \in \mathcal{G}$, the intersection $\Lambda_1 \cap \Lambda_2$ is either empty or a disjoint union of segments;
	\item if $\gamma$ is a reduced path between two vertices of some $\Lambda \in \mathcal{G}$ which is disjoint from some cycle of $\Lambda$, then $\gamma \subset \Lambda$.
\end{itemize}
If $\mathcal{G}$ is a collection of proper subgraphs, then $B_2(\Gamma)$ is relatively hyperbolic.
\end{thm}

In all our discussion dedicated to graph braid groups, we illustrate our criteria by giving concrete examples. Also, we leave several open questions which we think to be of interest in Section \ref{section:questions}.

\section{Preliminaries}

\subsection{Special cube complexes}\label{section:prel}

\noindent
A \textit{cube complex} is a CW-complex constructed by gluing together cubes of arbitrary (finite) dimension by isometries along their faces. Furthermore, it is \textit{nonpositively curved} if the link of any of its vertices is a simplicial \textit{flag} complex (i.e., $n+1$ vertices span a $n$-simplex if and only if they are pairwise adjacent), and \textit{CAT(0)} if it is nonpositively curved and simply connected. See \cite[page 111]{MR1744486} for more information.

\medskip \noindent
A fundamental feature of cube complexes is the notion of \textit{hyperplane}. Let $X$ be a nonpositively curved cube complex. Formally, a \textit{hyperplane} $J$ is an equivalence class of edges, where two edges $e$ and $f$ are equivalent whenever there exists a sequence of edges $e=e_0,e_1,\ldots, e_{n-1},e_n=f$ where $e_i$ and $e_{i+1}$ are parallel sides of some square in $X$. Notice that a hyperplane is uniquely determined by one of its edges, so if $e \in J$ we say that $J$ is the \textit{hyperplane dual to $e$}. Geometrically, a hyperplane $J$ is rather thought of as the union of the \textit{midcubes} transverse to the edges belonging to $J$. See Figure \ref{figure17}.
\begin{figure}
\begin{center}
\includegraphics[trim={0 13cm 10cm 0},clip,scale=0.4]{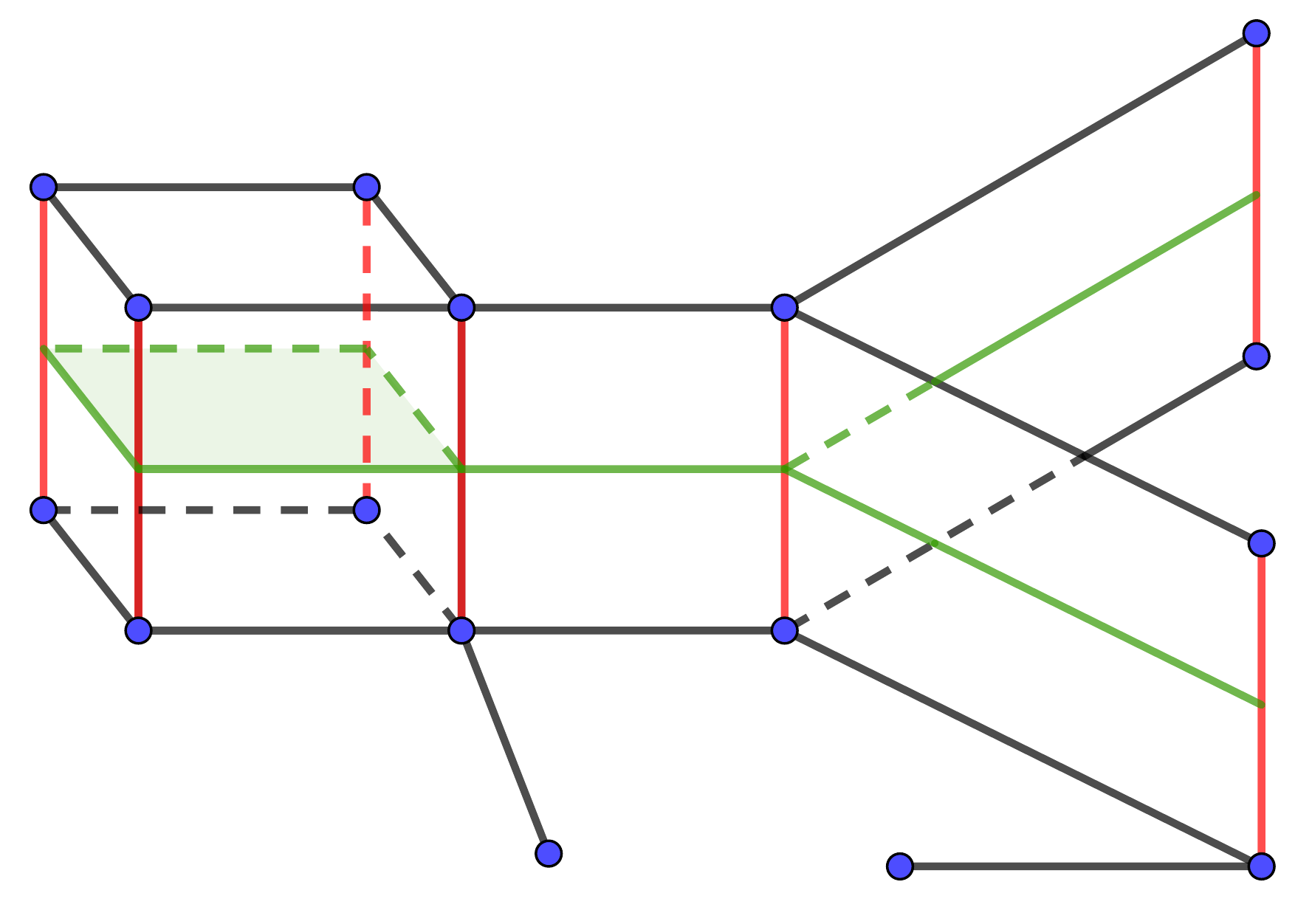}
\caption{A hyperplane (in red) and its geometric realisation (in green).}
\label{figure27}
\end{center}
\end{figure}
Similarly, one may define \textit{oriented hyperplanes} as classes of oriented edges. If $J$ is the hyperplane dual to an edge $e$ and if we fix an orientation $\vec{e}$, we will note $\vec{J}$ the oriented hyperplane dual to $\vec{e}$. It may be thought of as an \textit{orientation} of $J$, and we will note $- \vec{J}$ the opposite orientation of $J$. 

\begin{definition}
Let $X$ be a cube complex. The \emph{crossing graph} $\Delta X$ is the graph whose vertices are the hyperplanes of $X$ and whose edges link two transverse hyperplane. Similarly, the \emph{oriented crossing graph} $\vec{\Delta} X$ is the graph whose vertices are the oriented hyperplanes of $X$ and whose edges link two oriented hyperplanes whenever their underlying unordered hyperplanes are transverse.
\end{definition}

\noindent
Roughly speaking, \emph{special cube complexes} are nonpositively-curved cube complexes which do not contain ``pathological configurations'' of hyperplanes. Let us define precisely what these configurations are.

\begin{definition}
Let $J$ be a hyperplane with a fixed orientation $\vec{J}$. We say that $J$ is:
\begin{itemize}
	\item \textit{2-sided} if $\vec{J} \neq - \vec{J}$,
	\item \textit{self-intersecting} if there exist two edges dual to $J$ which are non-parallel sides of some square,
	\item \textit{self-osculating} if there exist two oriented edges dual to $\vec{J}$ with the same initial points or the same terminal points, but which do not belong to a same square.
\end{itemize}
Moreover, if $H$ is another hyperplane, then $J$ and $H$ are:
\begin{itemize}
	\item \textit{transverse} if there exist two edges dual to $J$ and $H$ respectively which are non-parallel sides of some square,
	\item \textit{inter-osculating} if they are transverse, and if there exist two edges dual to $J$ and $H$ respectively with one common endpoint, but which do not belong to a same square.
\end{itemize}
\end{definition}
\noindent
Sometimes, one refers 1-sided, self-intersecting, self-osculating and inter-osculating hyperplanes as \textit{pathological configurations of hyperplanes}. The last three configurations are illustrated on Figure \ref{figure17}.
\begin{figure}
\begin{center}
\includegraphics[trim={0, 19.5cm 1cm 0},clip,scale=0.45]{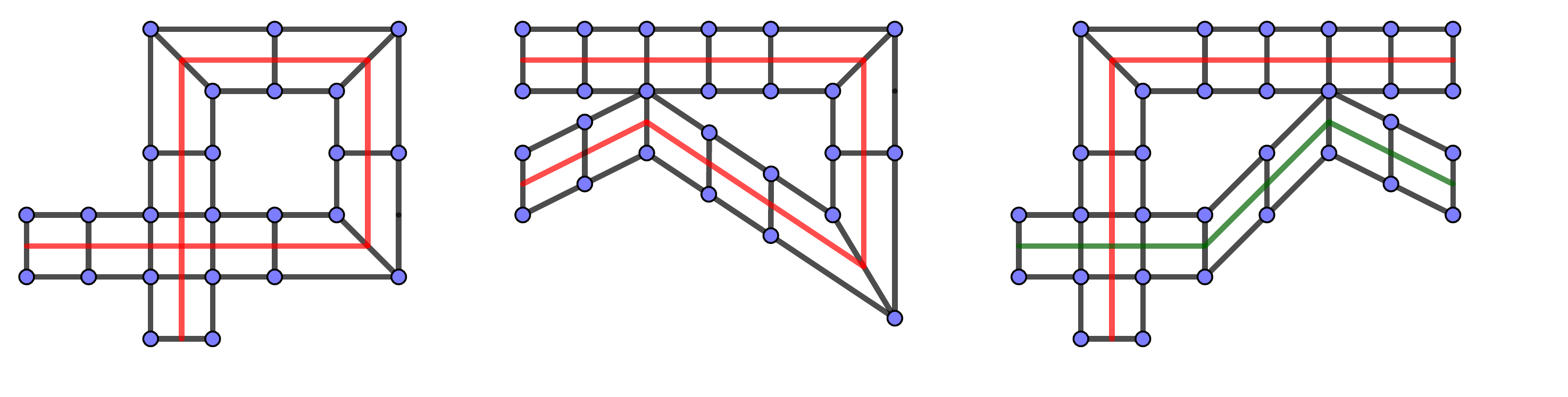}
\caption{From left to right: self-intersection, self-osculation, inter-osculation.}
\label{figure17}
\end{center}
\end{figure}

\begin{definition}
A \emph{special cube complex} is a nonpositively curved cube complex whose hyperplanes are two-sided and which does not contain self-intersecting, self-osculating or inter-osculating hyperplanes. A group which can be described as the fundamental group of a compact special cube complex is \emph{special}. A \emph{virtually special group} is a group which contains a finite-index subgroup which is cocompact special. 
\end{definition}

\noindent
We emphasize that, in this article, a special group is the fundamental group of a \emph{compact} special cube complex. To avoid ambiguity with fundamental groups of not necessarily compact special cube complexes, they are sometimes referred to as \emph{cocompact special groups} in the literature.

\subsection{Homotopy in cube complexes} 

\noindent
Let $X$ be a cube complex (not necessarily nonpositively-curved). For us, a \emph{path} in $X$ is the data of a sequence of successively adjacent edges. What we want to understand is when two such paths are homotopic (with fixed endpoints). For this purpose, we need to introduce the following elementary transformations. One says that:  
\begin{itemize}
	\item a path $\gamma \subset X$ contains a \emph{backtrack} if the word of oriented edges read along $\gamma$ contains a subword $ee^{-1}$ for some oriented edge $e$;
	\item a path $\gamma' \subset X$ is obtained from another path $\gamma \subset X$ by \emph{flipping a square} if the word of oriented edges read along $\gamma'$ can be obtained from the corresponding word of $\gamma$ by replacing a subword $e_1e_2$ with $e_2'e_1'$ where $e_1',e_2'$ are opposite oriented edges of $e_1,e_2$ respectively in some square of $X$. 
\end{itemize}
We claim that these elementary transformations are sufficient to determine whether or not two paths are homotopic. More precisely:

\begin{prop}\label{prop:cubehomotopy}
Let $X$ be a cube complex and $\gamma,\gamma' \subset X$ two paths with the same endpoints. The paths $\gamma,\gamma'$ are homotopic (with fixed endpoints) if and only if $\gamma'$ can be obtained from $\gamma$ by removing or adding backtracks and flipping squares. 
\end{prop}

\noindent
This statement follows from the fact that flipping squares provide the relations of the fundamental groupoid of $X$; see \cite[Statement 9.1.6]{BrownGroupoidTopology} for more details.

\subsection{Special colorings}\label{section:formalism}

\noindent
In this section, we review the formalism introduced in \cite{SpecialRH}. Our central definition is the following:

\begin{definition}\label{def:specialcolor}
Let $X$ be a cube complex. A \emph{special coloring} $(\Delta,\phi)$ is the data of a graph $\Delta$ and a coloring map 
$$\phi : V \left( \vec{\Delta} X \right) \to V(\Delta) \sqcup V(\Delta)^{-1}$$
satisfying the following conditions:
\begin{itemize}
	\item for every oriented hyperplane $J$, the equality $\phi(J^{-1})=\phi(J)^{-1}$ holds;
	\item two transverse oriented hyperplanes have adjacent colors;
	\item no two oriented hyperplanes adjacent to a given vertex have the same color;
	\item two oriented edges with same origin whose dual oriented hyperplanes have adjacent colors must generate a square. 
\end{itemize}
\end{definition}
\begin{figure}
\begin{center}
\includegraphics[trim={0, 19.5cm 23.5cm 0},clip,scale=0.43]{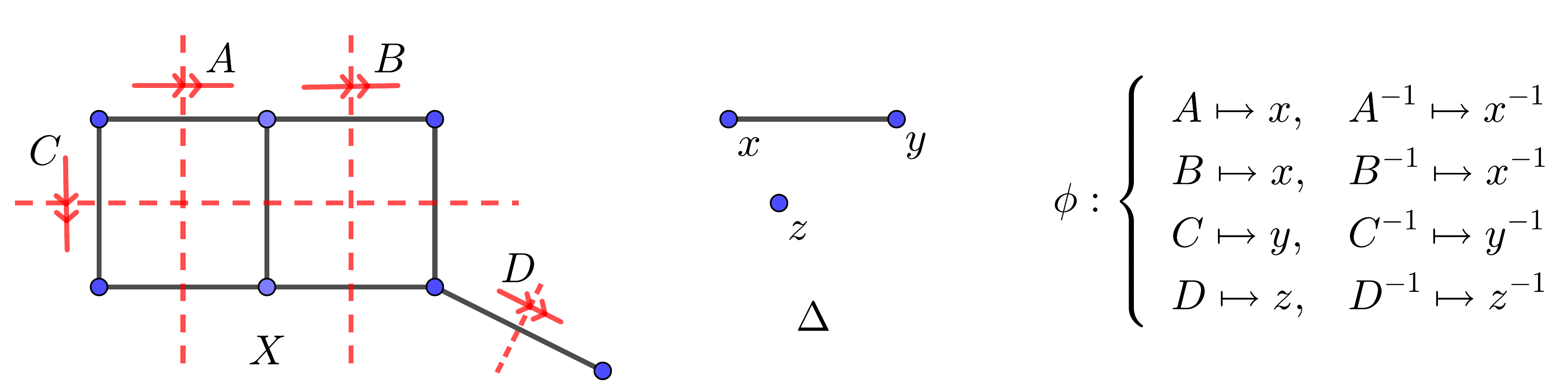}
\caption{Example of a special coloring.}
\label{SpecialColor}
\end{center}
\end{figure}

\noindent
See Figure \ref{SpecialColor} for an example. Essentially, a cube complex admits a special coloring if and only if it is special. We refer to \cite[Lemma 3.2]{SpecialRH} for a precise statement and a sketch of proof. 

\medskip \noindent
From now on, we fix a (not necessarily compact) special cube complex $X$ endowed with a special coloring $(\Delta,\phi)$. 

\medskip \noindent
Given a vertex $x_0 \in X$, a word $w=J_1 \cdots J_r$, where $J_1, \ldots, J_r \in V(\Delta) \sqcup V(\Delta)^{-1}$ are colors, is \emph{$x_0$-legal} if there exists a path $\gamma$ in $X$ starting from $x_0$ such that the oriented hyperplanes it crosses have colors $J_1, \ldots, J_r$ respectively. We say that the path $\gamma$ \emph{represents} the word $w$. 

\begin{fact}\label{fact:uniquepath}\emph{\cite[Fact 3.3]{SpecialRH}}
An $x_0$-legal word is represented by a unique path in $X$. 
\end{fact}

\noindent
The previous fact allows us to define the \emph{terminus} of an $x_0$-legal word $w$, denoted by $t(w)$, as the ending point of the unique path representing $w$. 

\medskip \noindent
Set $\mathcal{L}(X) = \{ x \text{-legal words} \mid x \in X\}$ the set of all legal words. (If $x_1,x_2 \in X$ are two distinct points, we consider the empty $x_1$-legal and $x_2$-legal words as distinct.) We consider the equivalence relation $\sim$ on $\mathcal{L}(X)$ generated by the following transformations:
\begin{description}
	\item[(cancellation)] if a legal word contains $JJ^{-1}$ or $J^{-1}J$, remove this subword;
	\item[(insertion)] given a color $J$, insert $JJ^{-1}$ or $J^{-1}J$ as a subword of a legal word;
	\item[(commutation)] if a legal word contains $J_1J_2$ where $J_1,J_2$ are two adjacent colors, replace this subword with $J_2 J_1$. 
\end{description}
So two $x$-legal words $w_1,w_2$ are equivalent with respect to $\sim$ if there exists a sequence of $x$-legal words 
$$m_1=w_1, \ m_2, \ldots, m_{r-1}, \ m_r=w_2$$ 
such that $m_{i+1}$ is obtained from $m_i$ by a cancellation, an insertion or a commutation for every $1 \leq i \leq r-1$. Define $\mathcal{D}(X)= \mathcal{L}(X)/ \sim$ as a set of \emph{diagrams}. The following observation allows us (in particular) to define the \emph{terminus} of a diagram as the terminus of one of the legal words representing it.

\begin{fact}\label{fact:transformations}\emph{\cite[Fact 3.4]{SpecialRH}}
Let $w'$ be an $x_0$-legal word obtained from another $x_0$-legal word $w$ by a cancellation / an insertion / a commutation. If $\gamma',\gamma$ are paths representing $w',w$ respectively, then $\gamma'$ is obtained from $\gamma$ by removing a backtrack / adding a backtrack / flipping a square. 
\end{fact}

\noindent
In the sequel, an \emph{$(x,y)$-diagram} will refer to a diagram represented by an $x$-legal word with terminus $y$, or just an \emph{$(x,\ast)$-diagram} if we do not want to specify its terminus. A diagram which is an $(x,x)$-diagram for some $x \in X$ is \emph{spherical}. 

\medskip \noindent
If $w$ is an $x_0$-legal word and $w'$ a $t(w)$-legal word, we define the \emph{concatenation} $w \cdot w'$ as the word $ww'$, which is $x_0$-legal since it is represented by the concatenation $\gamma \gamma'$ where $\gamma, \gamma'$ represent respectively $w,w'$. Because we have the natural identifications
\begin{table}[h]
	\centering
		\begin{tabular}{cccc}
			$\mathcal{L}(X)$ & $\leftrightarrow$ & paths in $X$ & (Fact \ref{fact:uniquepath}) \\ 
			$\sim$ & $\leftrightarrow$ & homotopy with fixed endpoints & (Fact \ref{fact:transformations}, Proposition \ref{prop:cubehomotopy}) 
		\end{tabular}
\end{table}

\noindent
it follows that the concatenation in $\mathcal{L}(X)$ induces a well-defined operation in $\mathcal{D}(X)$, making $\mathcal{D}(X)$ isomorphic to the fundamental groupoid of $X$. As a consequence, if we denote by $M(X)$ the Cayley graph of the groupoid $\mathcal{D}(X)$ with respect to the generating set $V(\Delta) \sqcup V(\Delta)^{-1}$, and, for every $x \in X$, $M(X,x)$ the connected component of $M(X)$ containing the trivial path $\epsilon(x)$ based at $x$, and $\mathcal{D}(X,x)$ the vertex-group of $\mathcal{D}(X)$ based at $\epsilon(x)$, then the previous identifications induce the identifications
\begin{table}[h]
	\centering
		\begin{tabular}{ccc}
			$\mathcal{D}(X)$ & $\leftrightarrow$ & fundamental groupoid of $X$  \\ 
			$\mathcal{D}(X,x)$ & $\leftrightarrow$ & $\pi_1(X,x)$  \\ 
			$M(X,x)$ & $\leftrightarrow$ & universal cover $\widetilde{X}^{(1)}$  \\ 
		\end{tabular}
\end{table}

\noindent
More explicitly, $\mathcal{D}(X,x)$ is the group of $(x,x)$-diagrams endowed with the concatenation, and $M(X,x)$ is the graph whose vertices are the $(x,\ast)$-diagrams and whose edges link two diagrams $w_1$ and $w_2$ if there exists some color $J$ such that $w_2=w_1J$. 

\medskip \noindent 
A more precise description of the identification between $M(X,x)$ and $\widetilde{X}^{(1)}$ is the following:
\begin{table}[h!]
	\centering
		\begin{tabular}{l}
			\hspace{5cm} $M(X,x) \longleftrightarrow \left( \widetilde{X}, \widetilde{x} \right)$ \\ \\ 
			\begin{tabular}{c} $(x,\ast)$-diagram represented \\ by an $x$-legal word $w$ \end{tabular}  $\mapsto$ \begin{tabular}{c} path $\gamma \subset X$ \\ representing $w$ \end{tabular}  $\mapsto$  \begin{tabular}{c} lift $\widetilde{\gamma} \subset \widetilde{X}$ of $\gamma$ \\ starting from $\widetilde{x}$ \end{tabular} $\mapsto$ \begin{tabular}{c} ending \\ point of $\widetilde{\gamma}$ \end{tabular} \\ \\
			\begin{tabular}{c} $(x,\ast)$-diagram represented by the \\ $x$-legal word corresponding to $\gamma$ \end{tabular} $\mapsfrom$ \begin{tabular}{c} image \\ $\gamma \subset X$ of $\gamma$ \end{tabular} $\mapsfrom$ \begin{tabular}{c} path $\widetilde{\gamma} \subset \widetilde{X}$ \\ from $\widetilde{x}$ to $y$ \end{tabular} $\mapsfrom$ $y$
		\end{tabular}
\end{table}

\noindent
It worth noticing that, if we color the oriented edges of $X$ as their dual hyperplanes, then the (oriented) edges of $\widetilde{X}$ are naturally labelled by colors and vertices of $\Delta$, just by considering their images in $X$. A consequence of the previous paragraph is that the generator labelling a given oriented edge of $M(X,x)$ is the same as the color labelling the corresponding edge of $\widetilde{X}$. 

\medskip \noindent
A diagram may be represented by several legal words. Such a word is \emph{reduced} if it has minimal length, i.e., it cannot be shortened by applying a sequence of cancellations, insertions and commutation. It is worth noticing that, in general, a diagram is not represented by a unique reduced legal word, but two such reduced words differ only by some commutations. (For instance, consider the homotopically trivial loop defined by the paths representing two of our reduced legal words, consider a disc diagram of minimal area bounded by this loop, and follow the proof of \cite[Theorem 4.6]{MR1347406}. Alternatively, use the embedding constructed in \cite[Section 4.1]{SpecialRH} (which does not use the present discussion) and conclude by applying the analogous statement which holds in right-angled Artin groups.) As a consequence, we can define the \emph{length} of a diagram as the length of any reduced legal word representing it. It is worth noticing that our length coincides with the length which is associated to the generating set $V(\Delta) \sqcup V(\Delta)^{-1}$ in the groupoid $\mathcal{D}(X)$. The next lemma follows from this observation.

\begin{lemma}
Let $D_1,D_2 \in M(X,x)$ be two $(x,\ast)$-diagrams. If $J_1 \cdots J_n$ is a reduced legal word representing $D_1^{-1}D_2$, then 
$$D_1, \ D_1J_1, \ D_1J_1J_2, \ldots, \ D_1J_1 \cdots J_n$$
is a geodesic from $D_1$ to $D_2$ in $M(X,x)$. Conversely, any geodesic between $D_1$ and $D_2$ arises in this way.
\end{lemma}

\noindent
Essentially by construction, one has:

\begin{fact}\label{fact:diagvsRAAG}
\emph{\cite[Fact 4.3]{SpecialRH}}
Let $X$ be a special cube complex, $x \in X$ a basepoint and $(\phi,\Delta)$ a special coloring of $X$. Two $x$-legal words of colors are equal in $\mathcal{D}(X,x)$ if and only if they are equal in the right-angled Artin group $A(\Delta)$. 
\end{fact}

\noindent
As an immediate consequence:

\begin{thm}\label{thm:embedRAAG}
\emph{\cite[Theorem 4.1]{SpecialRH}}
Let $X$ be a special cube complex and $x \in X$ a basepoint. If $(\phi,\Delta)$ is a special coloring of $X$, then the canonical map $\mathcal{D}(X,x) \to A(\Delta)$ induces an injective morphism. In particular, the fundamental group of $X$ embeds into the right-angled Artin group $A(\Delta)$. 
\end{thm}

\noindent
Let us also record the next lemma for future use.

\begin{lemma}\label{lem:commutationlegal}
\emph{\cite[Lemma 4.2]{SpecialRH}}
If $w'$ is a word obtained from an $x$-legal word $w$ by a commutation or a cancellation, then $w'$ is again an $x$-legal word.
\end{lemma}

\section{Graph braid groups and their cubical geometry}\label{section:gbgcubicalgeom}

\noindent
This first section is dedicated to basic definitions and properties of graph braid groups, essentially following \cite{GraphBraidGroups}, and to their cubical geometry. 

\begin{definition}
Let $\Gamma$ be a one-dimensional CW-complex and $n \geq 1$ an integer. The \emph{topological configuration space of $n$ points in $\Gamma$} is $C_n^{\mathrm{top}}(\Gamma)=\Gamma^n / D$ where $D= \{ (x_1, \ldots, x_n) \in \Gamma^n \mid x_i=x_j \ \text{for some $i \neq j$} \}$. The \emph{$n$th pure braid group of $\Gamma$ based at $\ast$}, denoted by $P_n(\Gamma, \ast)$, is the fundamental group of $C_n^{\mathrm{top}}(\Gamma)$ based at $\ast$. The symmetric group $\mathfrak{S}_n$ acts freely on $C_n^{\mathrm{top}}(\Gamma)$ by permuting the coordinates. Let $UC_n^{\mathrm{top}}(\Gamma)$ denote the quotient $C_n^{\mathrm{top}}(\Gamma)/ \mathfrak{S}_n$. The \emph{$n$th braid group of $\Gamma$ based at $\ast$}, denoted by $B_n(\Gamma, \ast)$, is the fundamental group of $UC_n^{\mathrm{top}}(\Gamma)$ based at $\ast$. 
\end{definition}

\noindent
A path in $C_n^{\mathrm{top}}(\Gamma)$ or in $UC_n^{\mathrm{top}}(\Gamma)$ can be thought of as moving continuously and without collision $n$ particles in $\Gamma$. It may be interesting to move the particles discretely, i.e., from a vertex to an adjacent vertex, in order to get a combinatorial model of the configuration space. More precisely:

\begin{definition}
Let $\Gamma$ be a one-dimensional CW-complex and $n \geq 1$ an integer. The \emph{combinatorial configuration space of $n$ points in $\Gamma$}, denoted by $C_n(\Gamma)$, is the subcomplex of $\Gamma^n$ containing all the cubes $\sigma_1 \times \cdots \times \sigma_n$ (where the $\sigma_i$'s are one-cells of $\Gamma$) such that, for every distinct $1 \leq i,j \leq n$, $\sigma_i$ and $\sigma_j$ do not share an endpoint. The symmetric group $\mathfrak{S}_n$ acts freely on $C_n(\Gamma)$ by permuting the coordinates. Let $UC_n(\Gamma)$ denote the quotient $C_n(\Gamma)/ \mathfrak{S}_n$. 
\end{definition}

\noindent
It is worth noticing that $UC_n^{\mathrm{top}}(\Gamma)$ depends only on the topological type of $\Gamma$ whereas $UC_n(\Gamma)$ depends heavily on the CW-structure of $\Gamma$. For instance, if $\Gamma$ is a circle with two vertices, then $UC_2^{\mathrm{top}}(\Gamma)$ is an annulus whereas $UC_2(\Gamma)$ is a single vertex. However, it was proved in \cite{GraphBraidGroups, PrueScrimshaw} that $UC_n^{\mathrm{top}}(\Gamma)$ and $UC_n(\Gamma)$ have the same homotopy type if the ``arcs'' of $\Gamma$ are sufficiently subdivided. More precisely:

\begin{prop}[\cite{PrueScrimshaw}]\label{prop:deformretract}
Let $n \geq 1$ be an integer and $\Gamma$ a one-dimensional CW-complex with at least $n$ vertices. If
\begin{itemize}
	\item each edge-path in $\Gamma$ connecting distinct vertices of degree at least three has length at least $n-1$;
	\item each homotopically non-trivial edge-path connecting a vertex to itself has length at least $n+1$,
\end{itemize}
then $UC_n^{\mathrm{top}}(\Gamma)$ deformation retracts onto $UC_n(\Gamma)$. 
\end{prop}

\noindent
An equivalent, but more direct, description of the cube complex $UC_n(\Gamma)$ is the following:
\begin{itemize}
	\item the vertices of $UC_n(\Gamma)$ are the subsets of $\Gamma^{(0)}$ of cardinality $n$;
	\item the edges of $UC_n(\Gamma)$ link two subsets $S_1$ and $S_2$ if the symmetric difference $S_1 \Delta S_2$ is a pair of adjacent vertices (so an edge of $UC_n(\Gamma)$ is naturally labelled by a one-cell of $\Gamma$);
	\item $n$ edges sharing a common endpoint generate an $n$-cube if the one-cells labelling them are pairwise disjoint.
\end{itemize}
This description implies easily that the link of every vertex of $UC_n(\Gamma)$ is a simplicial flag complex, which implies, as noticed in \cite{GraphBraidGroups}, that

\begin{prop}\label{prop:UCNPC}
Let $n \geq 1$ be an integer and $\Gamma$ a one-dimensional CW-complex containing more than $n$ vertices. Then $UC_n(\Gamma)$ is a disjoint union of nonpositively-curved cube complex. 
\end{prop}

\noindent
Given a one-dimensional CW-complex $\Gamma$ and an integer $n \geq 1$, we need to fix a basepoint in $UC_n(\Gamma)$, i.e., a subset $S \subset \Gamma^{(0)}$ of cardinality $n$, in order to define a graph braid group $B_n(\Gamma,S)$. However, up to isomorphism, our group depends only on the number vertices that $S$ contains in each connected components of $\Gamma$. Indeed, if $R$ is another subset of $\Gamma^{(0)}$ of cardinality $n$ such that $|R \cap \Lambda|=|S \cap \Lambda|$ for every connected component $\Lambda$ of $\Gamma$, then $R$ and $S$ belong to the same connected component of $UC_n(\Gamma)$, so that $B_n(\Gamma,S)$ and $B_n(\Gamma,R)$ turn out to be isomorphic. As a consequence, if $\Gamma$ is connected, the choice of the basepoint does not matter, so that the graph braid group can be simply denoted by $B_n(\Gamma)$. In fact, according to our next lemma, one can always assume that $\Gamma$ is connected. 

\begin{lemma}\label{lem:product}
Let $\Gamma$ be a one-dimensional CW-complex, $n \geq 2$ an integer and $S \in UC_n(\Gamma)$ an initial configuration. Suppose that $\Gamma= \Gamma_1 \sqcup \Gamma_2$. If our initial configuration $S$ has $p$ particles in $\Gamma_1$ and $q$ particles in $\Gamma_2$, then $B_n(\Gamma,S) \simeq B_p(\Gamma_1, S \cap \Gamma_1) \times B_q(\Gamma_2, S \cap \Gamma_2)$.
\end{lemma}

\begin{proof}
According to Proposition \ref{prop:deformretract}, up to subdividing $\Gamma$, we may suppose without loss of generality that the graph braid groups $B_n(\Gamma,S)$, $B_p(\Gamma_1,S \cap \Gamma_1)$ and $B_q(\Gamma_2, S \cap \Gamma_2)$ are respectively isomorphic to the fundamental groups of $UC_n(\Gamma,S)$, $UC_p(\Gamma_1, S \cap \Gamma_1)$ and $UC_q(\Gamma_2,S \cap \Gamma_2)$. Notice that the isomorphism
$$\left\{ \begin{array}{ccc} UC_n(\Gamma) & \to & \bigsqcup\limits_{k=0}^n UC_k(\Gamma_1) \times UC_{n-k}(\Gamma_2) \\ R & \mapsto & (R \cap \Gamma_1, R \cap \Gamma_2) \end{array} \right.$$
sends $UC_n(\Gamma,S)$ to $UC_p(\Gamma_1, S \cap \Gamma_1) \times UC_q(\Gamma_2, S \cap \Gamma_2)$. Hence the isomorphisms 
$$\begin{array}{lcl} B_n(\Gamma,S) & = & \pi_1( UC_n(\Gamma),S) \simeq \pi_1(UC_p(\Gamma_1) \times UC_q(\Gamma_2),S) \\ \\ & \simeq & B_p(\Gamma_1, S \cap \Gamma_1) \times B_q(\Gamma_2, S \cap \Gamma_2) \end{array}$$
which concludes our proof.
\end{proof}

\noindent
As a consequence of the previous lemma, since hyperbolic groups, and more generally acylindrically hyperbolic groups, do not split as direct products of infinite groups, one can always suppose that our one-dimensional CW-complex is connected when studying these kinds of groups.

\medskip \noindent
Now, let us focus on the hyperplanes of $UC_n(\Gamma)$. It is worth noticing that an oriented edge $E$ of $UC_n(\Gamma)$ is uniquely determined by the data of an oriented edge $e$ of $\Gamma$ together with a subset $S$ of $\Gamma^{(0)}$ of cardinality $n$ such that $o(e) \in S$ but $t(e) \notin S$: the edge $E$ links $S$ to $(S \backslash \{ o(e) \} ) \cup \{ t(e) \}$. For every such edge $(e,S)$, we denote by $[e,S]$ the hyperplanes dual to it. 

\begin{lemma}
Two oriented edges $(e_1,S_1)$ and $(e_2,S_2)$ of $UC_n(\Gamma)$ are dual to the same oriented hyperplane if and only if $e_1=e_2$ and $S_1 \backslash \{ o(e_1)\}, S_2 \backslash \{ o(e_2) \}$ belong to the same connected component of $C_{n-1}(\Gamma \backslash e )$. 
\end{lemma}

\begin{proof}
Suppose that $(e_1,S_1)$ and $(e_2,S_2)$ are dual to the same oriented hyperplane. So there exists a sequence of oriented edges
$$(\epsilon_1, \Sigma_1)= (e_1,S_1), \ (\epsilon_2, \Sigma_2), \ldots, (\epsilon_{n-1}, \Sigma_{n-1}), \ (\epsilon_n,\Sigma_n)= (e_2,S_2)$$
such that, for every $1 \leq i \leq n-1$, the edges $(\epsilon_i, \Sigma_i)$ and $(\epsilon_{i+1}, \Sigma_{i+1})$ are parallel in some square $Q_i$; let $(\eta_i,\Sigma_i)$ denote the edge of $Q_i$ which is adjacent to $(\epsilon_i,\Sigma_i)$ and which contains the starting of $(\epsilon_i,\Sigma_i)$, i.e., $\Sigma_i$. From the definition of the squares in $C_n(\Gamma)$, it follows that $\epsilon_i=\epsilon_{i+1}$, $\epsilon_i \cap \eta_i = \emptyset$ and $\Sigma_{i+1}= \left( \Sigma_i \backslash \{ o(\eta_i) \} \right) \cup \{ t(\eta_i) \}$ for every $1 \leq i \leq n-1$. Therefore, the equalities
$$e_1= \epsilon_1= \epsilon_2 = \cdots \epsilon_{n-1}= \epsilon_n=e_2$$
hold, and the sequence
$$S_1 \backslash \{ o(e_1) \}= \Sigma_1 \backslash \{ o(\epsilon_1) \}, \ \Sigma_2 \backslash \{o(\epsilon_2) \}, \ldots, \Sigma_n \backslash \{ o(\epsilon_n)\}= S_2 \backslash \{ o(e_2) \}$$
defines a path in $C_{n-1}(\Gamma \backslash e)$. 

\medskip \noindent
Conversely, let $e$ be an oriented edge of $\Gamma$ and $S_1,S_2$ two vertices of $C_{n-1}(\Gamma \backslash e)$ which belong to the same connected component. So there exists a path
$$\Sigma_1=S_1, \ \Sigma_2, \ldots, \Sigma_{n-1}, \ \Sigma_n = S_2$$
in $C_{n-1}( \Gamma \backslash e)$. For every $1 \leq i \leq n-1$, there exists an edge $\eta_i$ disjoint from $e$ such that $(\eta_i, \Sigma_i)$ links $\Sigma_i$ to $\Sigma_{i+1}$; notice that, since $\eta_i$ and $e$ are disjoint, the edges $(e,\Sigma_i \cup \{ o(e) \})$ and $(\eta_i, \Sigma_i \cup \{ o(e) \})$ generate a square in $C_n(\Gamma)$, such that $(e,\Sigma_{i+1} \cup \{ o(e) \})$ is the opposite edge of $(e,\Sigma_i \cup \{ o(e) \})$ in that square. From the sequence
$$(e,\Sigma_1 \cup \{o(e) \})=(e,S_1 \cup \{ o(e) \}), \ (e,\Sigma_2 \cup \{ o(e) \}), \ldots, (e, \Sigma_n \cup \{ o(e) \})=(e, S_2 \cup \{ o(e) \}),$$
it follows that the oriented edges $(e,S_1 \cup \{o(e) \})$ and $(e, S_2 \cup \{ o(e) \})$ of $C_n(\Gamma)$ are dual to the same oriented hyperplane, concluding the proof. 
\end{proof}

\noindent
Now, we are ready to define a special coloring of $UC_n(\Gamma)$. For that purpose, we denote by $\Delta$ the graph whose vertices are the edges of $\Gamma$ and whose edges link two disjoint edges of $\Gamma$, and we fix an orientation of the edges of $\Gamma$ in order to identify the set of oriented edges of $\Gamma$ with $E(\Gamma) \sqcup E(\Gamma)^{-1}$.

\begin{prop}\label{prop:UCspecialcoloring}
Let $\Gamma$ be a one-dimensional CW-complex and $n \geq 1$ an integer. If $\phi$ denotes the map $[e,S] \mapsto e$, then $(\Delta, \phi)$ is a special coloring of $UC_n(\Gamma)$. 
\end{prop}

\begin{proof}
The first point of Definition \ref{def:specialcolor} is clear from the construction of $\phi$. The other three points are also satisfied according the following observations.
\begin{itemize}
	\item If two hyperplanes are transverse, they must cross inside some square. Since two adjacent edges of $C_n(\Gamma)$ which generate a square must be labelled by disjoint edges, it follows that the images of our two hyperplanes under $\phi$ are adjacent in $\Delta$. 
	\item It is clear that two distinct edges of $C_n(\Gamma)$ starting from a common vertex must be labelled by different oriented edges of $\Gamma$. A fortiori, two oriented hyperplanes adjacent to a given vertex have different images under $\phi$.
	\item Consider two oriented edges of $C_n(\Gamma)$ starting from a common vertex such that the images under $\phi$ of the dual hyperplanes are adjacent in $\Delta$. This means that the edges of $\Gamma$ labelling our two edges of $C_n(\Gamma)$ are disjoint. From the definition of $C_n(\Gamma)$, it follows that our edges generate a square.
\end{itemize}
This concludes the proof of our proposition.  
\end{proof}

\begin{remark}
Following the proof of \cite[Lemma 3.2]{SpecialRH}, the combination of Propositions \ref{prop:UCNPC} and \ref{prop:UCspecialcoloring} provides an easy proof of the fact that $UC_n(\Gamma)$ is a disjoint union of special cube complexes, as noticed in \cite{CrispWiest}. 
\end{remark}

\noindent
Fixing a basepoint $S \in UC_n(\Gamma)$, we deduce from Section \ref{section:formalism} the following description of the universal cover $X_n(\Gamma,S)$ of the connected component of $UC_n(\Gamma)$ containing $S$, which we denote by $UC_n(\Gamma,S)$. 

\medskip \noindent
A word of oriented edges $e_1 \cdots e_n$ is $S$-legal if, for every $1 \leq i \leq n$, we have 
$$o(e_i) \in (S \backslash \{ o(e_1), \ldots, o(e_{i-1} \}) \cup \{ t(e_1), \ldots, t(e_{i-1}) \}$$ 
but 
$$t(e_i) \notin (S \backslash \{ o(e_1), \ldots, o(e_{i-1} \}) \cup \{ t(e_1), \ldots, t(e_{i-1}) \}.$$ 
The picture to keep in mind is the following. Start with the configuration $S$. The oriented edge $e_1$ acts on $S$ by moving the particle of $S$ at $o(e_1)$ to $t(e_1)$. Next, apply $e_2$ to this new configuration, and so on. The word $e_1 \cdots e_n$ is $S$-legal if all these operations are well-defined. The terminus of such a word is the final configuration we get. In this context, a cancellation amounts to removing a backtrack in the movement of one particle; an insertion adds such a backtrack; and a commutation switch the order of two independent moves. Therefore, an $(S_1,S_2)$-diagram can be thought of as $n$ particles moving discretely in $\Gamma$ without colliding, from a configuration $S_1$ to a configuration $S_2$, up to the previous operations.

\medskip \noindent
The one-skeleton of $X_n(\Gamma,S)$ can be identified with the graph whose vertices are the $(S,\ast)$-diagrams and whose edges link two diagrams if one can be obtained from the other by right-multiplying by an oriented edge of $\Gamma$. Moreover, the braid group $B_n(\Gamma,S)$ can be identified with the set of $(S,S)$-diagrams endowed with the concatenation. 

\medskip \noindent
As a direct consequence of Theorem \ref{thm:embedRAAG} and Proposition \ref{prop:UCspecialcoloring}, we are able to recover \cite[Theorem 2]{CrispWiest}.

\begin{prop}\label{prop:braidRAAG}
Let $\Gamma$ be one-dimensional CW-complex, $n \geq 1$ an integer and $S \in UC_n(\Gamma)$ a basepoint. The braid group $B_n(\Gamma,S)$ embeds into the right-angled Artin group $A(\Delta)$. 
\end{prop} 

\noindent
It is worth noticing that the graph $\Delta$ does not coincide with the crossing graph of the cube complex $UC_n(\Gamma)$, so that the previous embedding is different from the usual embedding one gets for special groups. In other words, a hyperplane $[e,S]$ of $UC_n(\Gamma)$ is not uniquely determined by the oriented edge $e$ of $\Gamma$. Figure \ref{figure6} provides a one-dimensional CW-complex $\Gamma$, an oriented edge of $\Gamma$, and two configurations $S_1$ and $S_2$, so that the hyperplanes $[e,S_1]$ and $[e,S_2]$ are distinct. 
\begin{figure}
\begin{center}
\includegraphics[trim={0 22cm 35cm 0},clip,scale=0.5]{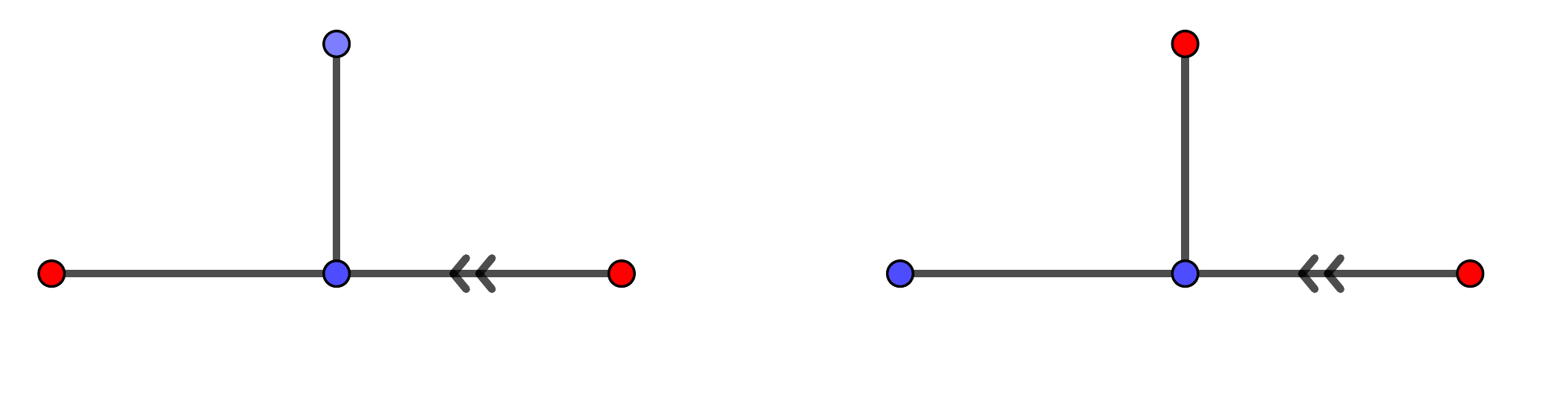}
\caption{Two distinct hyperplanes labelled by the same oriented edge.}
\label{figure6}
\end{center}
\end{figure}

\medskip \noindent
We conclude this section with an elementary observation about embeddings between graph braid groups. See also Question \ref{q:embedding}.

\begin{prop}\label{prop:braidembed}
Let $\Gamma$ be a connected one-dimensional CW-complex and $n \geq 1$ an integer. If $\Lambda$ is a connected induced subgraph of $\Gamma$, then $B_n(\Lambda)$ embeds into $B_n(\Gamma)$. Moreover, if $\Lambda$ is a proper subgraph, then $B_m(\Lambda)$ embeds into $B_n(\Gamma)$ for every $m \leq n$. 
\end{prop}

\begin{proof}
Fix an initial configuration $S \in UC_n(\Gamma)$ which is included in $\Lambda$, and consider the obvious map $UC_n(\Lambda,S) \to UC_n(\Gamma,S)$. By applying \cite[Theorem 1(2)]{CrispWiest}, it follows that this map is a local isometry. Because a local isometry between nonpositively-curved cube complexes turns out to be $\pi_1$-injective, we conclude that $B_n(\Lambda)$ embeds into $B_n(\Gamma)$.

\medskip \noindent
Now, suppose that $\Lambda \subset \Gamma$ is a proper subgraph and fix some integer $m \leq n$. Up to subdividing $\Gamma$, which does not affect the corresponding braid group according to Proposition \ref{prop:deformretract}, we may suppose without loss of generality that $\Gamma$ contains at least $n-m$ vertices which do not belong to $\Lambda$. Fix an initial configuration $S \in UC_n(\Gamma)$ satisfying $\# S \cap \Lambda=m$. As above, apply \cite[Theorem 1(2)]{CrispWiest} to show that the obvious map $UC_m(\Lambda, S \cap \Lambda) \to UC_n(\Gamma,S)$ is a local isometry. A fortiori, the braid group $B_m(\Lambda)$ embeds into $B_n(\Gamma)$. 
\end{proof}

\section{Properties of negative curvature}

\subsection{Gromov hyperbolicity}

\noindent
This section is dedicated to the characterisation of hyperbolic graph braid groups. We begin by introducing several classes of graphs.
\begin{itemize}
	\item A \emph{sun graph} is a graph obtained from a cycle by gluing rays to some vertices.
	\item A \emph{rose graph} is a graph obtained by gluing cycles and rays along a single vertex.
	\item A \emph{pulsar graph} is a graph obtained by gluing cycles along a fixed segment (not reduced to a single vertex) and rays to the endpoints of this segment.
\end{itemize}
See Figure \ref{figure3} for examples. We allow degenerate situations; for instance, a single cycle belongs to our three families of graphs. Our characterisation is the following:
\begin{figure}
\begin{center}
\includegraphics[trim={0 20cm 33.5cm 0},clip,scale=0.5]{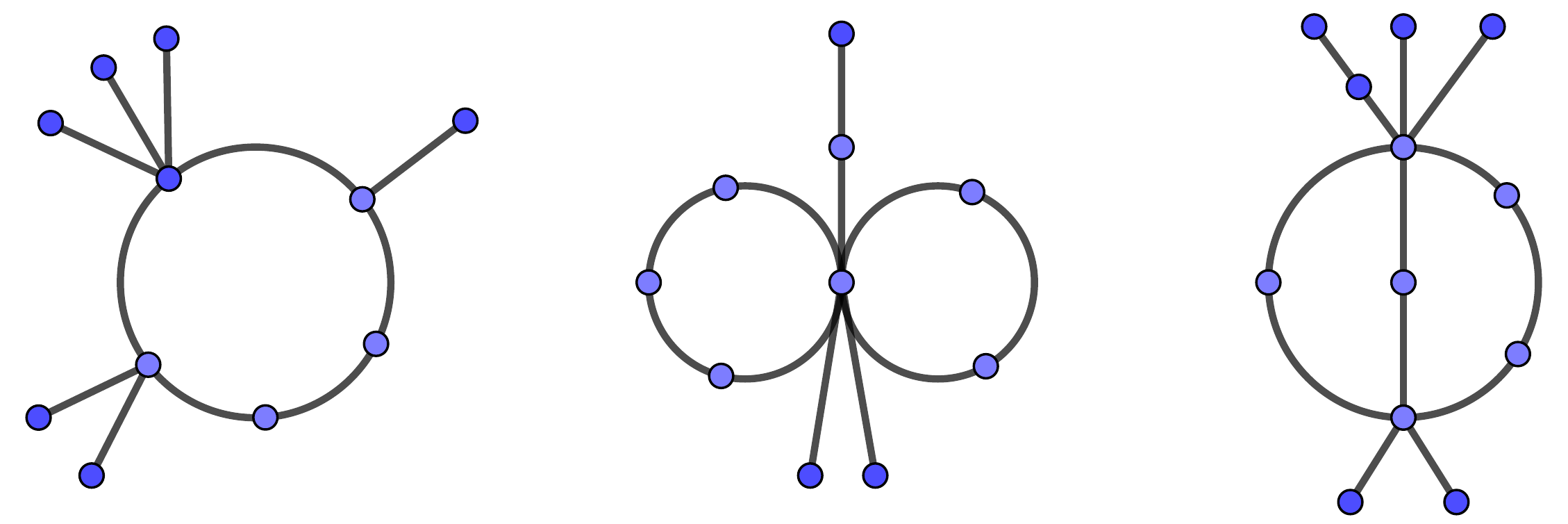}
\caption{From left to right: a sun graph, a rose graph, a pulsar graph.}
\label{figure3}
\end{center}
\end{figure}

\begin{thm}\label{thm:hypbraidgroup}
Let $\Gamma$ be a compact and connected one-dimensional CW-complex. 
\begin{itemize}
	\item The braid group $B_2(\Gamma)$ is hyperbolic if and only if $\Gamma$ does not contain a pair of disjoint induced cycles.
	\item The braid group $B_3(\Gamma)$ is hyperbolic if and only if $\Gamma$ is a tree, or a sun graph, or a rose graph, or a pulsar graph.
	\item For every $n \geq 4$, the braid group $B_n(\Gamma)$ is hyperbolic if and only if $\Gamma$ is a rose graph.
\end{itemize}
\end{thm}

\noindent
The proof of Theorem \ref{thm:hypbraidgroup} is based on \cite{SpecialRH} which shows that the fundamental group of a special cube complex is hyperbolic if and only if it does not contain $\mathbb{Z}^2$ as a subgroup. More precisely, we will use the following statement \cite[Fact 4.11]{SpecialRH}.

\begin{lemma}\label{lemma:nothyp}
Let $X$ be a cube complex and $(\Delta, \phi)$ a special coloring. If $\pi_1(X)$ is not hyperbolic, then there exist two commuting spherical diagrams $J_1 \cdots J_n$ and $H_1 \cdots H_m$ such that $J_i$ and $H_j$ are adjacent colors for every $1 \leq i \leq n$ and every $1 \leq j \leq m$. \qed
\end{lemma}

\noindent
Our second preliminary lemma determines when a graph braid group is trivial.

\begin{lemma}\label{lem:braidgroupnontrivial}
Let $\Gamma$ be a one-dimensional CW-complex. 
\begin{itemize}
	\item For every $S \in UC_1(\Gamma)$, the braid group $B_1(\Gamma,S)$ is non-trivial if and only if the connected component of $\Gamma$ containing $S$ also contains an induced cycle.
	\item For every $n \geq 2$ and every $S \in UC_n(\Gamma)$, the braid group $B_n(\Gamma,S)$ is non-trivial if and only if $\Gamma$ contains a connected component intersecting $S$ which contains an induced cycle, or a connected component whose intersection with $S$ has cardinality at least two and which contains a vertex of degree at least three. Alternatively, $B_n(\Gamma,S)$ is trivial if and only if the connected components of $\Gamma$ whose intersections with $S$ have cardinality one are trees, and those whose intersections with $S$ have cardinality at least two are segments.
\end{itemize}
\end{lemma}

\begin{proof}
The first assertion follows from the observation that the braid group $B_1(\Gamma,S)$ is isomorphic to the fundamental group of the connected component of $\Gamma$ containing $S$. Next, fix some $n \geq 2$. Write $\Gamma$ as 
$$A_1 \sqcup \cdots \sqcup A_p \sqcup B_1 \sqcup \cdots \sqcup B_q \sqcup C_1 \sqcup \cdots \sqcup C_r$$
where the $A_i$'s are the connected components of $\Gamma$ whose intersections with $S$ have cardinality one, where the $B_i$'s are the connected components of $\Gamma$ whose intersections with $S$ have cardinality at least two, and where the $C_i$'s are the connected components of $\Gamma$ which are disjoint from $S$. As a consequence of Lemma \ref{lem:product}, it is sufficient to show that, for every $1 \leq i \leq p$ and every $1 \leq j \leq q$, the braid groups $B_1(A_i)$ and $B_k(B_j)$, where $k \geq 2$, are trivial if and only if $A_i$ is a tree and $B_j$ is a segment. It is a consequence of the following facts.

\begin{fact}
For every connected one-dimensional CW-complex $\Gamma$, the braid group $B_1(\Gamma)$ is trivial if and only if $\Gamma$ is a tree.
\end{fact}

\noindent
Our fact follows directly from the observation that $B_1(\Gamma)$ is isomorphic to the fundamental group of $\Gamma$. 

\begin{fact}
Let $\Gamma$ be a connected one-dimensional CW-complex and $n \geq 2$ an integer. The braid group $B_n(\Gamma)$ is trivial if and only if $\Gamma$ is a segment.
\end{fact}

\noindent
Suppose that $\Gamma$ is not a segment. Then it must contain either an induced cycle or a vertex of degree at least three. If $\Gamma$ contains a cycle, either $\Gamma$ is itself a cycle, so that $B_n(\Gamma)$ must be infinite cyclic, or $\Gamma$ contains a proper cycle $C$, so that it follows from Proposition \ref{prop:braidembed} that $B_1(C) \simeq \mathbb{Z}$ embeds into $B_n(\Gamma)$. So, in this case, $B_n(\Gamma)$ is not trivial. Next, if $\Gamma$ contains a vertex of degree at least three, up to subdividing $\Gamma$ (which does not modify the corresponding braid group according to Proposition \ref{prop:deformretract}), one can apply Proposition \ref{prop:braidembed} to deduce that $B_2(T)$ embeds into $B_n(\Gamma)$, where is as in Figure~\ref{figure2}. Notice that $ca^{-1}bc^{-1}ab^{-1}$ is a non-trivial element of $B_2(T)$, so that $B_n(\Gamma)$ must be non-trivial. 
\begin{figure}
\begin{center}
\includegraphics[trim={0 23cm 47cm 0},clip,scale=0.5]{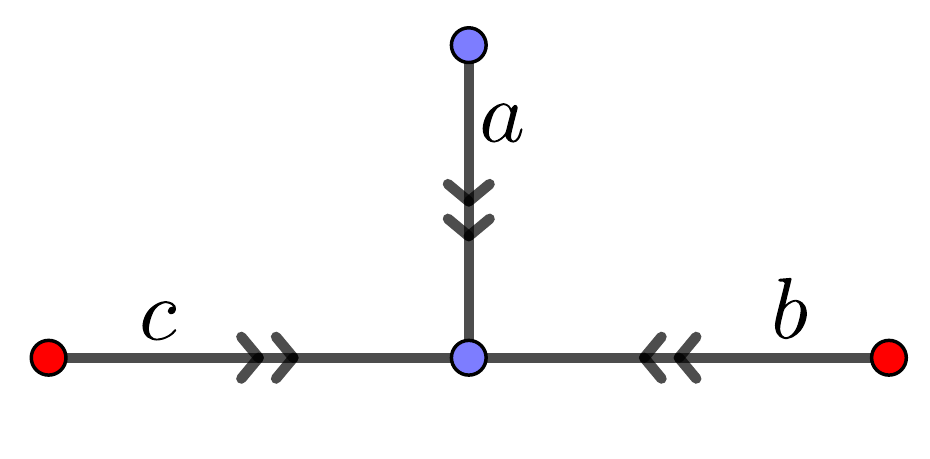}
\caption{}
\label{figure2}
\end{center}
\end{figure}

\medskip \noindent
Conversely, we leave as an exercise to show that the braid groups of a segment are trivial. 
\end{proof}

\noindent
We are now ready to prove Theorem \ref{thm:hypbraidgroup}. 

\begin{proof}[Proof of Theorem \ref{thm:hypbraidgroup}.]
Let $n \geq 2$ be an integer. For convenience, we fix an initial configuration $S \in UC_n(\Gamma)$. 

\medskip \noindent
Suppose that $B_n(\Gamma, S)$ is not hyperbolic. By applying Lemma \ref{lemma:nothyp}, we find two spherical diagrams $e_1 \cdots e_p, \epsilon_1 \cdots \epsilon_q \in \mathcal{D}(UC_n(\Gamma),S')$, for some $S' \in UC_n(\Gamma, S)$, such that $e_i \cap \epsilon_j= \emptyset$ for every $1 \leq i \leq p$ and $1 \leq j \leq q$. As a consequence, the subgraphs $\Lambda_1= e_1 \cup \cdots \cup e_p$ and $\Lambda_2= \epsilon_1 \cup \cdots \cup \epsilon_q$ are disjoint. Setting $S_1=S' \cap \Lambda_1$ (resp. $S_2 = S' \cap \Lambda_2$), notice that $e_1 \cdots e_p$ (resp. $\epsilon_1 \cdots \epsilon_q$) belongs to $\mathcal{D}(C_r(\Lambda_1),S_1)$ (resp. $\mathcal{D}(C_s(\Lambda_2),S_2)$), where $r$ (resp. $s$) denotes the cardinality of $S_1$ (resp. $S_2$). Thus, we have found a configuration $S'= S_1 \sqcup S_2 \in C_n(\Gamma, S)$, where $S_1,S_2$ have cardinalities $r,s$ respectively, and two subgraphs $\Lambda_1, \Lambda_2 \subset \Gamma$ containing $S_1,S_2$ respectively, such that $B_r(\Lambda_1,S_1)$ and $B_s(\Lambda_2,S_2)$ are non-trivial. 

\medskip \noindent
Conversely, if there exist a configuration $S'= S_1 \sqcup S_2 \in C_n(\Gamma, S)$, where $S_1,S_2$ have cardinalities $r,s$ respectively, and two subgraphs $\Lambda_1, \Lambda_2 \subset \Gamma$ containing $S_1,S_2$ respectively, such that $B_r(\Lambda_1,S_1)$ and $B_s(\Lambda_2,S_2)$ are non-trivial, then $B_n(\Gamma)$ must contain a subgroup isomorphic to $\mathbb{Z}^2$. Indeed, the subgroup of $B_n(\Gamma)$ generated by $B_r(\Lambda_1,S_1)$ and $B_s(\Lambda_2,S_2)$ is naturally isomorphic to $B_r(\Lambda_1,S_1) \times B_s(\Lambda_2,S_2)$ since $\Lambda_1$ and $\Lambda_2$ are disjoint. It follows that $B_n(\Gamma)$ cannot be hyperbolic.

\medskip \noindent
Thus, we have proved that $B_n(\Gamma,S)$ is hyperbolic if and only if there do not exist a configuration $S'= S_1 \sqcup S_2 \in C_n(\Gamma, \ast)$, where $S_1,S_2$ have cardinalities $r,s$ respectively, and two subgraphs $\Lambda_1, \Lambda_2 \subset \Gamma$ contains $S_1,S_2$ respectively, such that $B_r(\Lambda_1,S_1)$ and $B_s(\Lambda_2,S_2)$ are non-trivial. It follows from Lemma \ref{lem:braidgroupnontrivial} below that:
\begin{itemize}
	\item $B_2(\Gamma)$ is hyperbolic if and only if $\Gamma$ does not contain a pair of disjoint induced cycles;
	\item $B_3(\Gamma)$ is hyperbolic if and only if $\Gamma$ does not contain a pair of disjoint induced cycles nor a vertex of degree at least three which is disjoint from some induced cycle;
	\item for every $n \geq 4$, $B_n(\Gamma)$ is hyperbolic if and only if $\Gamma$ does not contain a pair of disjoint induced cycles, nor a vertex of degree at least three which is disjoint from some induced cycles, nor two distinct vertices of degree at least three.
\end{itemize} 
Suppose that $\Gamma$ is a graph satisfying all the conditions of our previous point. If $\Gamma$ does not contain a vertex of degree at least three, then it must be either a cycle or a segment. So suppose that $\Gamma$ contains a vertex of degree at least three. By assumption, $\Gamma$ contains a unique such vertex. Therefore, $\Gamma$ can be obtained by gluing along a single vertex some graphs without vertices of degree at least three, i.e., segments and cycles. It follows that $\Gamma$ must be a rose graph. Consequently, $B_n(\Gamma)$ is hyperbolic if and only if $\Gamma$ is a rose graph.

\medskip \noindent
Now, suppose that $\Gamma$ does not contain a pair of disjoint induced cycles nor a vertex of degree at least three which is disjoint from some induced cycle. If $\Gamma$ does not contain induced cycles, then $\Gamma$ is tree. If $\Gamma$ contains exactly one induced cycle, then $\Gamma$ must be a sun graph since all its vertices of degree at least three have to belong to this cycle. From now on, suppose that $\Gamma$ contains at least two induced cycles; let $C_1$ and $C_2$ be two such cycles. As a consequence, $\Gamma$ contains at least one vertex of degree at least three (which belongs to a path linking two cycles). If $\Gamma$ contains a single such vertex, then it follows from the previous paragraph that $\Gamma$ must be a rose graph. So suppose that $\Gamma$ contains at least two vertices of degree at least three. First, assume that $\Gamma$ contains exactly two such vertices, say $u$ and $v$. So $\Gamma$ can be constructed from $u$ and $v$ by gluing segments. Only two gluings are possible: identifying the two endpoints to $u$ and $v$ respectively, or identifying a single endpoint to $u$ or $v$. Indeed, identifying the two endpoints to $u$ (resp. $v$) is impossible since $v$ (resp. $u$) would be disjoint from the cycle thus created. Consequently, $\Gamma$ must be a pulsar graph. Finally, assume that $\Gamma$ contains at least three vertices of degree at least three, say $u,v,w$. By assumption, $\Gamma$ contains at least two induced cycles, say $C_1$ and $C_2$; moreover, $u,v,w$ must belong to these cycles. For every $x,y \in \{u,v,w\}$ and $i \in \{ 1,2 \}$, let $[x,y]_i$ denote the arc of $C_i$ between $x$ and $y$. Because $C_1 \neq C_2$, we can suppose without loss of generality that $[u,v]_1 \neq [u,v]_2$. Consequently, $[u,v]_1 \cup [u,v]_2$ contains an induced cycle which is disjoint from $w$, contradicting our assumptions. Consequently, $B_3(\Gamma)$ is hyperbolic if and only if $\Gamma$ is a tree, or a sun graph, or a rose graph, or a pulsar graph.
\end{proof}

\noindent
An interesting consequence of Theorem \ref{thm:hypbraidgroup}, combined with the next lemma, is that the (cohomologic or asymptotic) dimension of a hyperbolic graph braid group is at most three.

\begin{lemma}\label{lem:rosefree}
Let $\Gamma$ be a rose graph. For every $n \geq 2$, the braid group $B_n(\Gamma)$ is free. 
\end{lemma}

\begin{proof}
Let $o$ denote the center of $\Gamma$, and $e_1, \ldots, e_k$ its edges which are adjacent to $o$. Let $\mathcal{J}$ denote the collection of all the hyperplanes of the universal cover $X_n(\Gamma)$ of $UC_n(\Gamma)$ which are labelled by one of the edges $e_1, \ldots, e_k$. Because the edges labelling two transverse hyperplanes must be disjoint, it follows that $\mathcal{J}$ is a collection of pairwise disjoint hyperplanes; as a consequence, $\mathcal{J}$ induces an equivariant arboreal structure on $X_n(\Gamma)$. Notice that a connected component of $X_n(\Gamma)$ cut along $\mathcal{J}$ is naturally isometric to $X_n(\Gamma \backslash \{e_1, \ldots, e_k \},S)$ for some configuration $S \in UC_n(\Gamma \backslash \{e_1, \ldots, e_k \})$, and its stabiliser is naturally isomorphic to $B_n(\Gamma \backslash \{ e_1, \ldots, e_k \},S)$ for the same $S$. But the connected components of $\Gamma \backslash \{ e_1 , \ldots, e_k \}$ are all segments, so that we deduce that these stabilisers must be trivial. Consequently, $B_n(\Gamma)$ acts freely on the simplicial tree dual to the arboreal structure of $X_n(\Gamma)$ induced by $\mathcal{J}$, which implies that $B_n(\Gamma)$ must be free. 
\end{proof}

\begin{remark}\label{remark:criterionfree}
An argument similar to the previous one also implies \cite[Proposition~5.5]{PresentationsGraphBraidGroups}, namely: if $\Gamma$ is a connected one-dimensional CW-complex which contains a vertex belonging to all its induced cycles, then $B_2(\Gamma)$ is free. 
\end{remark}

\begin{ex}\label{ex1:GraphHyp}
The braid group $B_n(K_m)$ is hyperbolic if and only if $n=1$, or $n=2$ and $m \leq 5$, or $m \leq 3$. We already know that $B_1(K_m) \simeq \pi_1(K_m)$ is a free group of rank $(m-1)(m-2)/2$. Notice that if $m \leq 3$ then $K_m$ is either a single vertex, a segment or a cycle, so that $B_n(K_m)$ is either trivial or infinite cyclic. Therefore, the only interesting braid groups in our family are $B_2(K_4)$ and $B_2(K_5)$. As a consequence of \cite[Example 5.1]{GraphBraidGroups}, $B_2(K_5)$ is the fundamental group of a closed non orientable surface of genus six. 
\end{ex}

\begin{ex}\label{ex2:GraphHyp}
The braid group $B_n(K_{p,q})$, where we suppose that $p \leq q$, is hyperbolic if and only if $n=1$, or $n=2$ and $p \leq 3$, or $n=3$ and $p \leq 2$, or $n \geq 4$ and $p=q=2$, or $n \geq 4$ and $p=1$. In this family, the only braid group which might not be free have the form $B_2(K_{3,n})$ and $B_3(K_{2,n})$ where $n \geq 3$. We do not know whether or not these groups are actually free. (Notice that $K_{2,n}$ is a pulsar graph for every $n \geq 1$.) But at least one of them is not free: as a consequence of \cite[Example 5.2]{GraphBraidGroups}, $B_2(K_{3,3})$ is the fundamental group of a closed non orientable surface of genus four. 
\end{ex}

\subsection{Acylindrical hyperbolicity}

\noindent
In this section, our goal is to show that essentially all the graph braid groups turn out to be acylindrically hyperbolic. More precisely, we want to prove:

\begin{thm}\label{thm:braidgroupacyl}
Let $\Gamma$ be a connected compact one-dimensional CW-complex and $n \geq 2$ an integer. The braid group $B_n (\Gamma)$ is either cyclic or acylindrically hyperbolic. 
\end{thm}

\noindent
Our argument is based on the following criterion:

\begin{prop}\label{prop:specialcentraliseracyl}
Let $G$ be the fundamental group of a compact special cube complex. Assume that $G$ is not cyclic. Then $G$ is acylindrically hyperbolic if and only it contains a non-trivial element whose centraliser is cyclic. 
\end{prop}

\begin{proof}
As a consequence of \cite[Corollary 6.2]{SpecialRankOne}, combined with \cite{arXiv:1112.2666}, $G$ is acylindrically hyperbolic if it contains a non-trivial element whose centraliser is cyclic. Conversely, every acylindrically hyperbolic group contains an infinite-order element whose centraliser is virtually cyclic according to \cite[Corollary 6.6]{DGO}. Because $G$ is torsion-free, we deduce that, whenever $G$ is acylindrically hyperbolic, it must contain a non-trivial element whose centraliser is cyclic. 
\end{proof}

\noindent
Now, we need to introduce some terminology. Let $X$ be a special cube complex and $(\Delta, \phi)$ a special coloring of $X$. A diagram is \emph{cyclically reduced} if it cannot be written as a reduced legal word starting with a color and ending with its inverse. Any diagram $D$ is conjugate in $\mathcal{D}(X)$ to a unique cyclically reduced diagram: indeed, the analogue statement known for right-angled Artin groups can be transferred to diagrams according to Fact \ref{fact:diagvsRAAG}. This unique diagram is the \emph{cyclic reduction} of $D$. Going back to $X=UC_n(\Gamma)$, fix an $(S,\ast)$-diagram $D \in \mathcal{D}(X)$, where $S \in UC_n(\Gamma)$ is some initial configuration, and let $R$ denote its cyclic reduction. Writing $R$ as an $S'$-legal word of oriented edges $e_1 \cdots e_k$, where $S' \in UC_n(\Gamma)$ is some other initial configuration, we define the \emph{support} of $D$, denoted by $\mathrm{supp}(D)$, as the \emph{full support} of $R$, i.e., the subgraph $e_1 \cup \cdots \cup e_k \subset \Gamma$; and the \emph{set of particles} of $D$, denoted by $\mathrm{part}(D)$, as $S' \cap \mathrm{supp}(D)$. So, if $D$ is cyclically reduced and if it is thought of as a motion of particles in $\Gamma$, then $\mathrm{part}(D)$ is the set of particles which really moves, and $\mathrm{supp}(D)$ is the subgraph of $\Gamma$ in which they are confined. These two sets are used in our next statement to construct elements whose centralisers are infinite cyclic. 

\begin{prop}\label{prop:cycliccentraliser}
Let $\Gamma$ be a one-dimensional CW-complex, $n \geq 1$ an integer and $S \in UC_n(\Gamma)$ an initial configuration. A non-trivial element $g \in B_n (\Gamma,S)$ has a cyclic centraliser if $\mathrm{supp}(g)$ is connected and if $\mathrm{part}(g)$ has cardinality $n$. 
\end{prop}

\begin{proof}
Let $g \in B_n(\Gamma,S)$ be a non-trivial element. According to Fact \ref{fact:diagvsRAAG} and to \cite[Centralizer Theorem]{ServatiusCent}, there exist words of oriented edges $a,h_1, \ldots, h_k$ and integers $m_1, \ldots, m_k \in \mathbb{Z} \backslash \{ 0 \}$ such that our element $g$ (thought of as a diagram) can be written as an $S$-legal word $ah_1^{m_1} \cdots h_k^{m_k}a^{-1}$ where $h_1 \cdots h_k$ is the cyclic reduction of $g$ and where any oriented edge of $h_i$ is disjoint from any oriented edge of $h_j$ for every distinct $1 \leq i,j \leq n$; moreover, the centraliser of $g$ is the set of $(S,S)$-legal words of oriented edges which can be written as $ah_1^{r_1} \cdots h_k^{r_k} h a^{-1}$, where $r_1, \ldots, r_k \in \mathbb{Z}$ are integers and where $h$ is a word of oriented edges which are disjoint from the edges of the $h_i$'s. 

\medskip \noindent
Notice that the support of $g$ is the disjoint union of the full supports of the $h_i$'s, hence $k=1$ since $\mathrm{supp}(g)$ is connected by assumption. So $g=ah_1^{m_1}a^{-1}$. Let $p \in B_n(\Gamma,S)$ be an element of the centraliser of $g$. From the previous paragraph, we know that $p$ can be written as a legal word $ah_1^{r_1}ha^{-1}$ for some integer $r_1 \in \mathbb{Z}$ and some word $h$ of oriented edges which are disjoint disjoint from those of $h_1$. Notice that $\mathrm{supp}(p)$ is the disjoint of the full supports of $h$ and $h_1$. But the full support of $h_1$ coincides with the support of $g$, which has cardinality $n$. Since a full support must have cardinality at most $n$, it follows that $h$ is trivial, i.e., $p=ah_1^{r_1}a^{-1}$. 

\medskip \noindent
Thus, we have proved that, for every element $p \in B_n(\Gamma,S)$ of the centraliser of $g$, its power $p^{m_1}$ belongs to $\langle g \rangle$. This proves that $\langle g \rangle$ has finite index in the centraliser of $g$, and a fortiori that the centraliser of $g$ is infinite cyclic. 
\end{proof}

\begin{proof}[Proof of Theorem \ref{thm:braidgroupacyl}.]
If $B_n(\Gamma)$ is trivial, there is nothing to prove, so we suppose that $B_n(\Gamma)$ is non-trivial. According to Lemma \ref{lem:braidgroupnontrivial}, $\Gamma$ must contain either a cycle or a vertex of degree at least three. 

\medskip \noindent
Suppose that $\Gamma$ contains a cycle $C$. Without loss of generality, we may suppose that our initial configuration $S \in UC_n(\Gamma)$ is included in $C$. Let $g \in B_n(\Gamma,S)$ be the element ``rotating all the particles around $C$''. For instance, if $C$ is as in Figure \ref{figure5} (i), then $g=cba$. Since $g$ does not contain an oriented edge and its inverse, it is clearly cyclically reduced. Moreover, its support is $C$ and its set of particles has full cardinality. We deduce from Proposition~\ref{prop:cycliccentraliser} that the centraliser of $g$ is infinite cyclic.
\begin{figure}
\begin{center}
\includegraphics[trim={0 20.5cm 38cm 0},clip,scale=0.5]{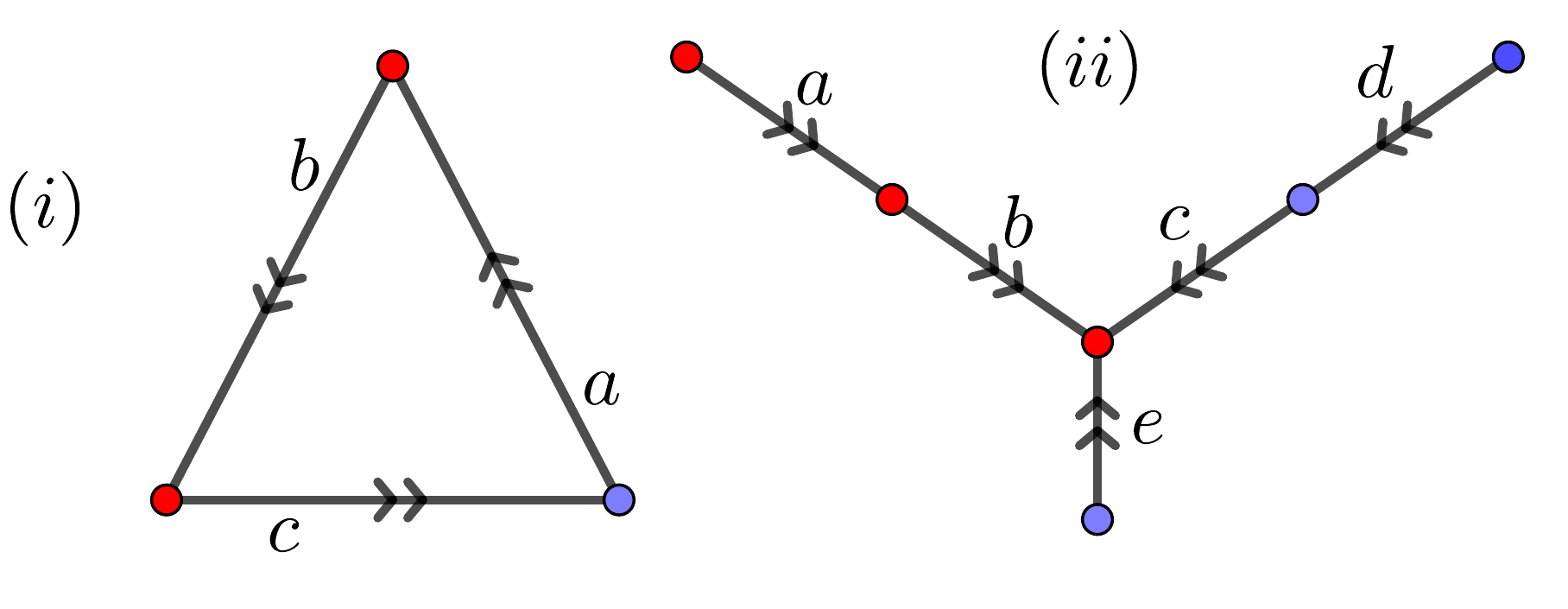}
\caption{}
\label{figure5}
\end{center}
\end{figure}

\medskip \noindent
Next, suppose that $\Gamma$ contains a vertex of degree at least three. As a consequence, $\Gamma$ contains a tripod $T$. Up to subdividing $\Gamma$ (which does not modify the braid group according to Proposition \ref{prop:deformretract}) and extracting subgraph of $T$, suppose that $T$ is isomorphic to $[-n+1,n-1] \times \{ 0 \} \cup \{0 \} \times [0,1]$. For convenience, let $k$ denote the vertex $(k,0) \in T$ for every $-n+1 \leq k \leq n-1$, and let $p$ denote the vertex $(0,1)$. Without loss of generality, suppose that the $n$ vertices of our initial configuration $S \in UC_n(\Gamma)$ are $-n+1, \ldots, 0$. Now, let $g \in B_n(\Gamma,S)$ be the element corresponding to the following motion of particles: 
\begin{itemize}
	\item move the particle at $0$ to $p$;
	\item for $k=-1,-2, \ldots, -n+1$, move the particle at $k$ to $n+k$;
	\item move the particle at $p$ to $-n+1$;
	\item for $k=1,2, \ldots, n-1$, move the particle at $k$ to $-n+k+1$. 
\end{itemize}
For instance, if $n=3$, $T$ is as in Figure \ref{figure5} (ii) and
$$g= e^{-1}bc^{-1}d^{-1}abc^{-1}eb^{-1}a^{-1}cb^{-1}dc.$$
Notice that the first and last letters of any $S$-legal word of edges representing $g$ must be $[0,p]$ and $[0,1]$ respectively. Therefore, $g$ is cyclically reduced. It follows that the support of $g$ is $T$ and that it set of particles has full cardinality. We deduce from Proposition~\ref{prop:cycliccentraliser} that the centraliser of $g$ is infinite cyclic.

\medskip \noindent
Thus, we have proved that $B_n(\Gamma)$ contains an element whose centraliser is infinite cyclic. We conclude from Proposition \ref{prop:specialcentraliseracyl} that $B_n(\Gamma)$ is either cyclic or acylindrically hyperbolic. 
\end{proof}

\begin{remark}
We suspect that, even when $\Gamma$ is not compact, an application of the criterion \cite[Theorem 4.17]{article3} implies that an element as in Proposition \ref{prop:cycliccentraliser} induces a contracting isometry on the universal cover $X_n(\Gamma)$. This would show that, for any connected one-dimensional CW-complex $\Gamma$ and for any integer $n \geq 1$, the graph braid group $B_n(\Gamma)$ is either cyclic or acylindrically hyperbolic. However, such a graph braid group would not be finitely generated in general. Indeed, if $B_n(\Gamma)$ is finitely generated then there exists a compact subcomplex $\Lambda \subset \Gamma$ such that $B_n(\Gamma) \simeq B_n(\Lambda)$. As often in geometric group theory, we are mainly interested in finitely generated groups, so we do not pursue the generalisation mentioned above. 
\end{remark}

\noindent
Theorem \ref{thm:braidgroupacyl} stresses out the following question: when is a graph braid group cyclic? We already know from Lemma \ref{lem:braidgroupnontrivial} when it is trivial, so it remains to determine when it is infinite cyclic. This is the goal of our next lemma.

\begin{lemma}\label{lem:braidcyclic}
Let $\Gamma$ be a connected one-dimensional CW-complex.
\begin{itemize}
	\item The braid group $B_2(\Gamma)$ is infinite cyclic if and only if $\Gamma$ is a cycle or a star with three arms.
	\item For every $n \geq 3$, the braid group $B_n(\Gamma)$ is infinite cyclic if and only if $\Gamma$ is a cycle. 
\end{itemize}
\end{lemma}

\begin{proof}
We begin by proving that various graph braid groups are not infinite cyclic. Our general criterion will follow easily from these observations. 
\begin{figure}
\begin{center}
\includegraphics[trim={0 10.5cm 30cm 0},clip,scale=0.5]{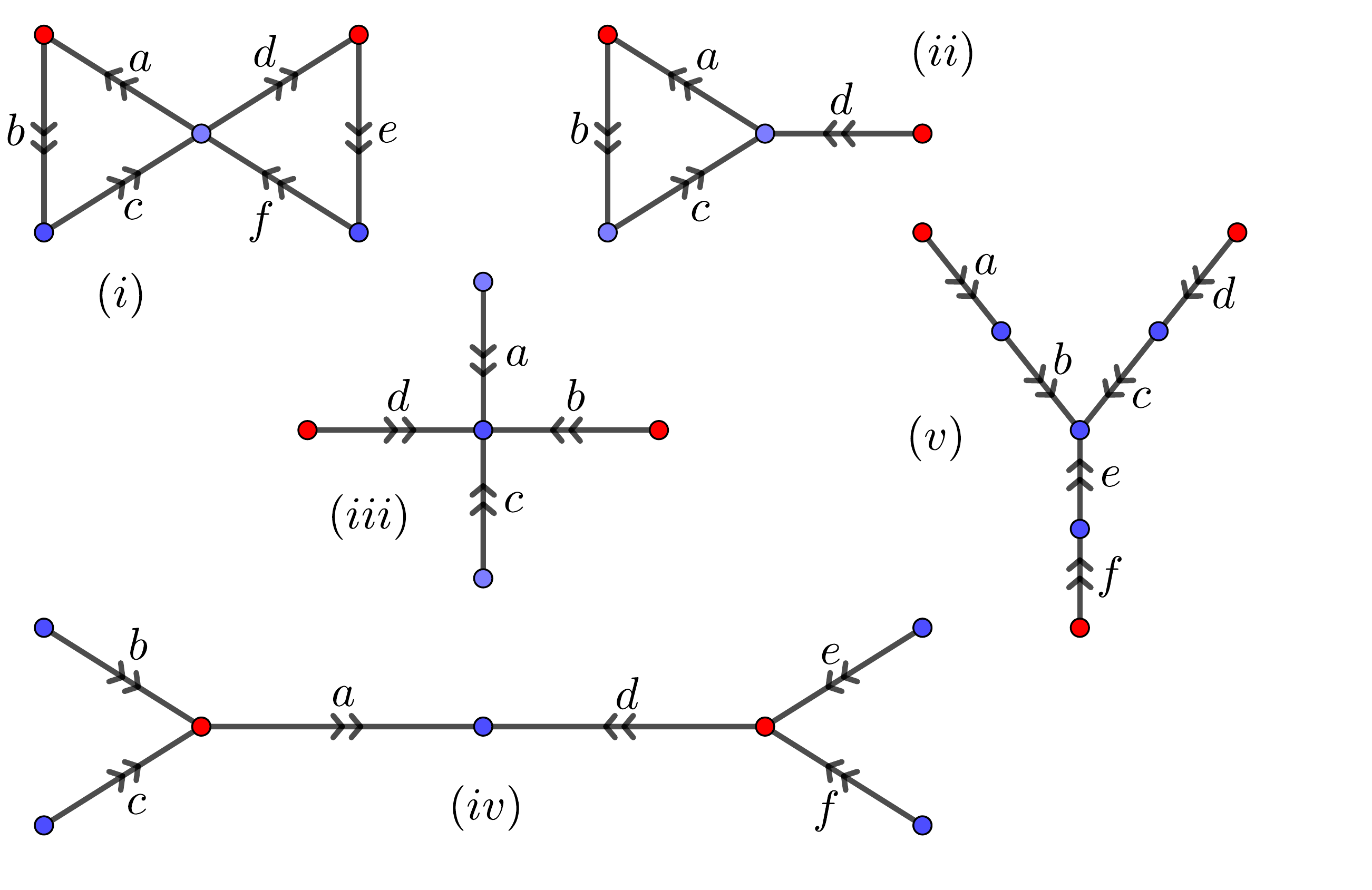}
\caption{}
\label{figure1}
\end{center}
\end{figure}

\begin{fact}\label{fact1}
If $\Gamma$ is a bouquet of two circles, then $B_2(\Gamma)$ is not infinite cyclic.
\end{fact}

\noindent
Consider $\Gamma$ as in Figure \ref{figure1} (i). Set $g=bca$ and $h=efd$. For every $n,m \in \mathbb{Z} \backslash \{ 0 \}$, one has $g^n=(bca)^n \neq (efd)^m = h^m$. This implies that $B_2(\Gamma)$ is not infinite cyclic. 

\begin{fact}\label{fact2}
If $\Gamma$ is the wedge of a cycle and a segment, then $B_2(\Gamma)$ is not infinite cyclic.
\end{fact}

\noindent
Consider $\Gamma$ as in Figure \ref{figure1} (ii). Set $g= bca$ and $h=dbacd^{-1}$. For every $n,m \in \mathbb{Z} \backslash \{ 0 \}$, one has $g^n=(bca)^n \neq d(bac)^m d^{-1}= h^m$. This implies that $B_2(\Gamma)$ is not infinite cyclic.

\begin{fact}\label{fact3}
If $\Gamma$ is a star with four arms, then $B_2(\Gamma)$ is not infinite cyclic.
\end{fact}

\noindent
Consider $\Gamma$ as in Figure \ref{figure1} (iii). Set $g= da^{-1}bd^{-1}ab^{-1}$ and $h=dc^{-1}bd^{-1}cb^{-1}$. For every $n,m \in \mathbb{Z} \backslash \{ 0 \}$, one has $g^n= (da^{-1}bd^{-1}ab^{-1})^n \neq (dc^{-1}bd^{-1}cb^{-1})^m =h^m$. This implies that $B_2(\Gamma)$ is not infinite cyclic. 

\begin{fact}\label{fact4}
If $\Gamma$ is a union of two segments whose middle points are linked by a third segments, then $B_2(\Gamma)$ is not infinite cyclic.
\end{fact}

\noindent
Consider $\Gamma$ as in Figure \ref{figure1} (iv). Set $g= dc^{-1}a^{-1}b^{-1}cad^{-1}b$ and $h= f^{-1}ad^{-1}e^{-1}fda^{-1}e$. For every $n,m \in \mathbb{Z} \backslash \{ 0 \}$, one has 
$$g^n = ( dc^{-1}a^{-1}b^{-1}cad^{-1}b)^n \neq ( f^{-1}ad^{-1}e^{-1}fda^{-1}e)^m = h^m.$$ 
This implies that $B_2(\Gamma)$ is not infinite cyclic. 

\medskip \noindent
Now, let $\Gamma$ be a connected one-dimensional CW-complex such that $B_2(\Gamma)$ is infinite cyclic. According to Fact \ref{fact4}, $\Gamma$ contains at most one vertex of degree at least three. Therefore, $\Gamma$ must be a union of $n$ cycles and $m$ segments glued along a single vertex (i.e., $\Gamma$ is a rose graph). Two cases may happen. First, if $\Gamma$ contains a cycle (i.e., $n \geq 1$) then it follows from Facts \ref{fact1} and \ref{fact2} that necessarily $n=1$ and $m=0$, i.e., $\Gamma$ is a cycle. Secondly, if $\Gamma$ does not contain any cycle (i.e., $n=0$), then $\Gamma$ must be a star with $m$ arms. But we know from Fact \ref{fact3} that necessarily $m \leq 3$. Since $B_2(\Gamma)$ is trivial if $m \leq 2$, we conclude that $\Gamma$ must be a star with three arms. Conversely, we left as an exercise to show that $B_2(\Gamma)$ is indeed infinite cyclic if $\Gamma$ is a cycle or a star with three arms.

\medskip \noindent
Next, let $\Gamma$ be a one-dimensional CW-complex and $n \geq 3$ an integer such that $B_n(\Gamma)$ is infinite cyclic. If $\Gamma$ contains a vertex of degree at least three then, up to subdividing $\Gamma$ (which does not affect the braid group $B_n(\Gamma)$ according to Proposition \ref{prop:deformretract}), we can deduce from Proposition \ref{prop:braidembed} that $B_n(\Gamma)$ contains $B_2(T)$ where $T$ is a star with three arms. According to the fact below, this is impossible. Therefore, $\Gamma$ must be either a cycle or a segment. We know that $B_n(\Gamma)$ is trivial if $\Gamma$ is a segment, so $\Gamma$ must be a cycle. Conversely, we leave as an exercise to show that $B_n(\Gamma)$ is infinite cyclic if $\Gamma$ is a cycle. 

\begin{fact}
If $\Gamma$ is a star with three arms, then $B_3(\Gamma)$ is not infinite cyclic. 
\end{fact}

\noindent
Consider $\Gamma$ as in Figure \ref{figure1} (v). Set the elements $g= abe^{-1}dcb^{-1}a^{-1}ec^{-1}d^{-1}$ and $h=abc^{-1}feb^{-1}a^{-1}ce^{-1}f^{-1}$. For every $n,m \in \mathbb{Z} \backslash \{ 0 \}$, one has
$$g^n = (abe^{-1}dcb^{-1}a^{-1}ec^{-1}d^{-1})^n \neq ( abc^{-1}feb^{-1}a^{-1}ce^{-1}f^{-1} )^m = h^m.$$
This implies that $B_2(\Gamma)$ is not infinite cyclic.
\end{proof}

\subsection{Relative hyperbolicity}\label{section:RH}

\noindent
In this section, we are interested in determining when a graph braid group turns out to be relatively hyperbolic. We first focus on toral relatively hyperbolic groups, i.e., groups hyperbolic relative to free abelian subgroups. Our main result in this direction is:

\begin{thm}\label{thm:braidRHabelian}
Let $\Gamma$ be a connected compact one-dimensional CW-complex. 
\begin{itemize}
	\item The braid group $B_2(\Gamma)$ is toral relatively hyperbolic if and only if $\Gamma$ does not contain three pairwise disjoint induced cycles.
	\item The braid group $B_3(\Gamma)$ is toral relatively hyperbolic if and only if $\Gamma$ does not contain an induced cycle disjoint from two other induced cycles; nor a vertex of degree at least four disjoint from an induced cycle; nor a segment between two vertices of degree three which is disjoint from an induced cycle; nor a vertex of degree three which is disjoint from two induced cycles; nor two disjoint induced cycles one of those containing a vertex of degree three.
	\item The braid group $B_4(\Gamma)$ is toral relatively hyperbolic if and only if $\Gamma$ is a rose graph, or a segment linking two vertices of degree three, or a cycle containing two vertices of degree three, or two cycles glued along a non-trivial segment.
	\item For every $n \geq 5$, the braid group $B_n(\Gamma)$ is toral relatively hyperbolic if and only if $\Gamma$ is a rose graph. If so, $B_n(\Gamma)$ is a free group. 
\end{itemize}
\end{thm}

\noindent
Our argument is based on the criterion provided by \cite[Theorem 5.16]{SpecialRH}, namely:

\begin{thm}[\cite{SpecialRH}]\label{thm:RHspecialAbelian}
The fundamental group of a compact special cube complex is toral relatively hyperbolic if and only if it does not contain $\mathbb{F}_2 \times \mathbb{Z}$ as a subgroup.
\end{thm}

\noindent
In order to apply Theorem \ref{thm:RHspecialAbelian}, one needs to find subgroups isomorphic to $\mathbb{F}_2 \times \mathbb{Z}$. As a first step we determine when a graph braid group contains a non-abelian free subgroup. 

\begin{lemma}\label{lem:freeinbraid}
Let $\Gamma$ be a compact one-dimensional CW-complex, $n \geq 1$ an integer, and $S \in UC_n(\Gamma)$ an initial configuration. The following conditions are equivalent:
\begin{itemize}
	\item[(i)] the braid group $B_n(\Gamma,S)$ contains a non-abelian free subgroup;
	\item[(ii)] the braid group $B_n(\Gamma,S)$ is not free abelian;
	\item[(iii)] there exists a connected component of $\Gamma$ containing one particle of $S$ which contains at least two cycles; or a connected component containing two particles of $S$ which is not a segment, a cycle or a star with three arms; or a connected component containing at least three particles of $S$ which is neither a segment nor a cycle. 
\end{itemize}
\end{lemma}

\begin{proof}
The equivalence $(i) \Leftrightarrow (ii)$ is a consequence of Tits' alternative for right-angled Artin groups \cite{MR634562}, combined with Proposition \ref{prop:braidRAAG}. If $\Lambda_1, \ldots, \Lambda_r$ denote the connected components of $\Gamma$, we know from Lemma \ref{lem:product} that
$$B_n(\Gamma,S) \simeq B_{n_1}(\Lambda_1) \times \cdots \times B_{n_r}(\Lambda_r),$$
where $n_i = \# S \cap \Lambda_i$ for every $1 \leq i \leq r$. Therefore, the braid group $B_n(\Gamma,S)$ is free abelian if and only if all the $B_{n_i}(\Lambda_i)$'s are free abelian themselves, which occurs if and only if they are cyclic according to Theorem \ref{thm:braidgroupacyl}. The equivalence $(ii) \Leftrightarrow (iii)$ follows from Lemmas \ref{lem:braidgroupnontrivial} and \ref{lem:braidcyclic}. 
\end{proof}

\noindent
Now, we are ready to determine when a graph braid group contains a subgroup isomorphic to $\mathbb{F}_2 \times \mathbb{Z}$. Our theorem will be a direct consequence of the following proposition combined with Theorem \ref{thm:RHspecialAbelian}. 

\begin{prop}\label{prop:FxZinBraid}
Let $\Gamma$ be a connected compact one-dimensional CW-complex. 
\begin{itemize}
	\item The braid group $B_2(\Gamma)$ contains $\mathbb{F}_2 \times \mathbb{Z}$ if and only if $\Gamma$ contains an induced cycle which is disjoint from two other induced cycles.
	\item The braid group $B_3(\Gamma)$ contains $\mathbb{F}_2 \times \mathbb{Z}$ if and only if $\Gamma$ contains an induced cycle disjoint from two other induced cycles; or a vertex of degree at least four disjoint from an induced cycle; or a segment between two vertices of degree three which is disjoint from an induced cycle; or a vertex of degree three which is disjoint from two induced cycles; or two disjoint induced cycles one of those containing a vertex of degree three.
	\item The braid group $B_4(\Gamma)$ contains $\mathbb{F}_2 \times \mathbb{Z}$ if and only if $\Gamma$ is not a rose graph, nor a segment linking two vertices of degree three, nor a cycle containing two vertices of degree three, nor two cycles glued along a non-trivial segment.
	\item For every $n \geq 5$, the braid group $B_n(\Gamma)$ contains $\mathbb{F}_2 \times \mathbb{Z}$ if and only if $\Gamma$ is not a rose graph.
\end{itemize}
\end{prop}

\begin{proof}
Let $n \geq 2$ be an integer. We claim that $B_n(\Gamma)$ contains a subgroup isomorphic to $\mathbb{F}_2 \times \mathbb{Z}$ if and only if $\Gamma$ contains two disjoint induced subgraphs $\Lambda_1,\Lambda_2 \subset \Gamma$ such that $B_n(\Lambda_1, S \cap \Lambda_1)$ contains a non-abelian free subgroup and such that $B_n(\Lambda_2,S \cap \Lambda_2)$ is non-trivial for some initial configuration $S \in UC_n(\Gamma)$. 

\medskip \noindent
Fix an initial configuration $S \in UC_n(\Gamma)$. Suppose first that $B_n(\Gamma,S)$ contains a subgroup isomorphic to $\mathbb{F}_2 \times \mathbb{Z}$. Let $a,b \in B_n(\Gamma,S)$ be two $S$-legal words of edges generating $\mathbb{F}_2$, and $z \in \mathbb{Z}$ an $S$-legal word of edges generating $\mathbb{Z}$. Up to conjugating $z$ (and changing the initial configuration), we may suppose that $z$ is cyclically reduced. As a consequence of Proposition \ref{prop:braidRAAG} and \cite[Centralizer Theorem]{ServatiusCent}, one can write $z$ as a reduced product $h_1^{n_1} \cdots h_r^{n_r}$ such that any edge of the word $h_i$ is disjoint from any edge of the word $h_j$ if $i \neq j$; and $a$ and $b$ respectively as $h_1^{p_1} \cdots h_r^{p_r}h$ and $h_1^{q_1} \cdots h_r^{q_r}k$ such that the edges of the words $h$ and $k$ are disjoint from the edges of the $h_i$'s. As a consequence of Lemma~\ref{lem:commutationlegal}, for every $j \geq 1$, one has
$$a^j = (h_1^{p_1} \cdots h_r^{p_r}h)^j = h^jh_1^{jp_1} \cdots h_r^{jp_r};$$
a fortiori, $h^j$ is an $S$-legal word. Because the terminus of a word may take only finitely values, there exists two distinct integers $i,j \geq 1$ such that $h^i$ and $h^j$ have the same terminus, so that $h^ih^{-j}=h^{i-j}$ must belong to $B_n(\Gamma,S)$. A similar argument works for $k$. Therefore, there exist two integers $p,q \geq 1$ such that $h^p$ and $k^q$ belong to $B_n(\Gamma,S)$. Let $\Lambda_1$ be the subgraph of $\Gamma$ generated by the edges of the words $h$ and $k$, and $\Lambda_2$ the subgraph generated by the edges of the word $z$. Notice that $\Lambda_1$ and $\Lambda_2$ are disjoint, and that $z \in B_n(\Lambda_2,S \cap \Lambda_2)$ and $h^p,k^q \in B_n(\Lambda_1, S \cap \Lambda_1)$. A fortiori, $B_n(\Lambda_2, S \cap \Lambda_2)$ is non-trivial, and $B_n(\Lambda_1,S \cap \Lambda_1)$ contains a non-abelian free subgroup. Indeed, because the powers $a^p$ and $b^q$ do not commute (since $a$ and $b$ generate a non-abelian free subgroup), necessarily $h^p$ and $k^q$ cannot commute, so that $\langle h^p,k^q \rangle$ must be a free group as a consequence of \cite{MR634562} (which applies thanks to Proposition \ref{prop:braidRAAG}). This is the desired conclusion. 

\medskip \noindent
The converse follows from Proposition \ref{prop:braidembed} and Lemma \ref{lem:product}, concluding the proof of our claim. 

\medskip \noindent
Consequently, $B_2(\Gamma)$ contains $\mathbb{F}_2 \times \mathbb{Z}$ if and only if there exist two disjoint (connected) subgraphs $\Lambda_1,\Lambda_2 \subset \Gamma$ such that $B_1(\Lambda_1)$ contains a non-abelian free subgroup and $B_1(\Lambda_2)$ is non-trivial. The conclusion follows from Lemmas \ref{lem:braidgroupnontrivial} and \ref{lem:freeinbraid}. 

\medskip \noindent
Next, $B_3(\Gamma,S)$ contains $\mathbb{F}_2 \times \mathbb{Z}$ if and only if if there exist two disjoint subgraphs $\Lambda_1,\Lambda_2 \subset \Gamma$ such that either $B_1(\Lambda_1,S \cap \Lambda_1)$ contains a non-abelian free subgroup and $B_2(\Lambda_2, S \cap \Lambda_2)$ is non-trivial, or $B_2(\Lambda_1,S \cap \Lambda_1)$ contains a non-abelian free subgroup and $B_1(\Lambda_2, S \cap \Lambda_2)$ is non-trivial. According to Lemmas \ref{lem:braidgroupnontrivial} and \ref{lem:freeinbraid}, the former situation is equivalent to: $\Gamma$ contains an induced cycle which is disjoint from two other induced cycles or it contains a vertex of degree at least three which is disjoint from two induced cycles; and the latter situation is equivalent to: $\Gamma$ contains an induced cycle which is disjoint from two other induced cycles or from a connected subgraph which not a segment, a cycle or a star with three arms. The desired conclusion follows from the next observation. 

\begin{fact}
A connected graph is distinct from a segment, a cycle and a star with three arms if and only if it contains a vertex of degree at least four, or if its contains two distinct vertices of degree three, or if it is a cycle containing a vertex of degree three. 
\end{fact}

\noindent
Let $\Lambda$ be a connected graph. If $\Lambda$ does not contain any vertex of degree at least three, then $\Lambda$ must be either a segment or a cycle. If $\Lambda$ contains at least two distinct vertices of degree at least three, we are done, so suppose that $\Lambda$ contains a unique vertex of degree at least three. If this vertex has degree at least four, we are done, so suppose that it has degree exactly three. Two cases may happen: either $\Lambda$ is a star with three arms, or $\Lambda$ is a cycle containing a vertex of degree three.

\medskip \noindent
Now, notice that $B_4(\Gamma)$ contains $\mathbb{F}_2 \times \mathbb{Z}$ if $\Gamma$ contains at least three vertices of degree three or if it contains it contains two distinct vertices respectively of degrees at least three and four. In the former case, $\Gamma$ contains two disjoint subgraphs $\Lambda_1$ and $\Lambda_2$ such that $\Lambda_1$ is a segment between two vertices of degree three and $\Lambda_2$ a star with three arms. By fixing an initial configuration $S \in UC_4(\Gamma)$ satisfying $\#S \cap \Lambda_1 \geq 2$ and $\# S \cap \Lambda_2 \geq 2$, it follows from Lemmas \ref{lem:braidgroupnontrivial} and \ref{lem:freeinbraid} that $B_4(\Lambda_1, S \cap \Lambda_1)$ contains a non-abelian free subgroup and that $B_4(\Lambda_2, S \cap \Lambda_2)$ is non-trivial. In the latter case, $\Gamma$ contains two disjoint subgraphs $\Lambda_1$ and $\Lambda_2$ such that $\Lambda_1$ is a star with at least three arms and $\Lambda_2$ a star with at least four arms. By fixing an initial configuration $S \in UC_4(\Gamma)$ satisfying $\#S \cap \Lambda_1 \geq 2$ and $\# S \cap \Lambda_2 \geq 2$, it follows from Lemmas \ref{lem:braidgroupnontrivial} and \ref{lem:freeinbraid} that $B_4(\Lambda_1, S \cap \Lambda_1)$ is non-trivial and that $B_4(\Lambda_2, S \cap \Lambda_2)$contains a non-abelian free subgroup. This proves our claim. From now on, suppose that $\Gamma$ contains at most two vertices of degree at least three, and that it does not contain two vertices respectively of degrees at least three and four. 

\medskip \noindent
If $\Gamma$ is a tree, only two cases may happen: $\Gamma$ is a star (and in particular a rose graph) or a segment between two vertices of degree three. If $\Gamma$ contains a unique cycle, only three cases may happen: $\Gamma$ is reduced to a cycle, or $\Gamma$ is a cycle which contains a vertex of degree at least three, or $\Gamma$ is a cycle with two vertices of degree three. Notice that, in the first two cases, $\Gamma$ is a rose graph. From now on, we suppose that $\Gamma$ contains at least two induced cycles. 

\medskip \noindent
Next, notice that, if $\Gamma$ contains a vertex of degree at least three which is disjoint from a cycle containing a vertex of degree at least three, then $B_4(\Gamma)$ contains $\mathbb{F}_2 \times \mathbb{Z}$ as a subgroup. Indeed, when it is the case, $\Gamma$ contains two disjoint subgraphs $\Lambda_1$ and $\Lambda_2$ where $\Lambda_1$ is an induced cycle with a vertex of degree three and $\Lambda_2$ a star with three arms. By fixing an initial configuration $S \in UC_4(\Gamma)$ satisfying $\#S \cap \Lambda_1 \geq 2$ and $\# S \cap \Lambda_2 \geq 2$, it follows from Lemmas \ref{lem:braidgroupnontrivial} and \ref{lem:freeinbraid} that $B_4(\Lambda_1, S \cap \Lambda_1)$ contains a non-abelian free subgroup and that $B_4(\Lambda_2, S \cap \Lambda_2)$ is non-trivial. The desired conclusion follows. From now on, we suppose that $\Gamma$ does not contain any vertex of degree at least three which is disjoint from a cycle containing a vertex of degree at least three.

\medskip \noindent
As a consequence, the induced cycles of $\Gamma$ must pairwise intersect: otherwise, $\Gamma$ would contain two disjoint induced cycles joined by a segment. Let $C_1$ and $C_2$ be two induced cycles of $\Gamma$. According to our previous observation, $C_1$ and $C_2$ intersect. We distinguish two cases. First, suppose that the intersection $C_1 \cap C_2$ is not reduced to a singel vertex. As a consequence of our assumptions, $C_1 \cap C_2$ must be a segment whose endpoints are two vertices of degree three. Because $\Gamma$ cannot contain any other vertex of degree at least three, it follows that $\Gamma=C_1 \cup C_2$, i.e., $\Gamma$ is a union of two induced cycles glued along a non-trivial segment. Next, suppose that $C_1 \cap C_2$ is reduced to a single vertex. A fortiori, $\Gamma$ contains a vertex of degree at least four, so that no other vertex can have degree at least three. It follows that $\Gamma$ must be a rose graph. 

\medskip \noindent
Thus, we have proved that, if $B_4(\Gamma)$ does not contain $\mathbb{F}_2 \times \mathbb{Z}$, then $\Gamma$ must be
\begin{itemize}
	\item a rose graph;
	\item or a segment linking two vertices of degree three;
	\item or a cycle containing two vertices of degree three;
	\item or two cycles glued along a non-trivial segment.
\end{itemize}
Conversely, we claim that if $\Gamma$ is one of these graphs then $B_4(\Gamma)$ does not contain $\mathbb{F}_2 \times \mathbb{Z}$. Let $S \in UC_4(\Gamma)$ be an initial configuration and $\Lambda_1,\Lambda_2 \subset \Gamma$ two subgraphs such that $B_4(\Lambda_1,S \cap \Lambda_1)$ is non-trivial. We distinguish two cases. First, suppose that $\# S \cap \Lambda_1=1$. Because $B_4(\Lambda_1,S \cap \Lambda_1)$ is non-trivial, $\Lambda_1$ must contain a cycle, so that $\Lambda_2$ must be included in the complement of a cycle. By looking at our different graphs, this implies that $\Lambda_2$ must be a disjoint union of segments, so that $B_4(\Lambda_2, S \cap \Lambda_2)$ has to be trivial. Next, suppose that $\# S \cap \Lambda_1 \geq 2$. If $\Lambda_1$ contains a cycle, the previous argument shows that $B_4(\Lambda_2, S \cap \Lambda_2)$ is trivial, so suppose that $\Lambda_1$ does not contain any cycle. Because $B_4(\Lambda_1, S \cap \Lambda_1)$ is non-trivial, necessarily $\Lambda_1$ must contain a vertex of degree at least three, so that $\Lambda_2$ is included in the complement of a vertex of degree at least three. By looking at our different graphs, we deduce that $\Lambda_2$ either is a disjoint union of segments, so that $B_4(\Lambda_2, S \cap \Lambda_2)$ is trivial, or is included in a star with three arms. In the latter case, $\Lambda_2$ is either a disjoint union of segments, such that $B_4(\Lambda_2,S \cap \Lambda_2)$ is trivial, or a disjoint union of segments with a star having three arms, so that $B_4(\Lambda_2, S \cap \Lambda_2)$ cannot contain a non-abelian free subgroup according to Lemma~\ref{lem:freeinbraid} since $\# S \cap \Lambda_2 \leq 4- \# S \cap \Lambda_1 \leq 2$. As a consequence of the equivalence proved at the beginning of this proof, it follows that $B_4(\Gamma)$ does not contain $\mathbb{F}_2 \times \mathbb{Z}$ as a subgroup. 

\medskip \noindent
Finally, let $n \geq 5$ be an integer. Suppose first that $\Gamma$ contains at least two distinct vertices of degree at least three. Then $\Gamma$ contains two disjoint subgraphs $\Lambda_1$ and $\Lambda_2$ isomorphic to stars with three arms. By fixing an initial configuration $S \in UC_n(\Gamma)$ satisfying $\#S \cap \Lambda_1 \geq 3$ and $\# S \cap \Lambda_2 \geq 2$, it follows from Lemmas \ref{lem:braidgroupnontrivial} and \ref{lem:freeinbraid} that $B_4(\Lambda_1, S \cap \Lambda_1)$ contains a non-abelian free subgroup and that $B_4(\Lambda_2, S \cap \Lambda_2)$ is non-trivial. Therefore, $B_n(\Gamma)$ contains $\mathbb{F}_2 \times \mathbb{Z}$ as a subgroup. Next, if $\Gamma$ contains at most one vertex of degree at least three, then it must be a rose graph. In this case, we know from Lemma \ref{lem:rosefree} that $B_n(\Gamma)$ is a free group so that $B_n(\Gamma)$ cannot contain $\mathbb{F}_2 \times \mathbb{Z}$ as a subgroup. 
\end{proof}

\begin{proof}[Proof of Theorem \ref{thm:braidRHabelian}.]
Our theorem follows directly from Proposition \ref{prop:FxZinBraid}, Theorem~\ref{thm:RHspecialAbelian} and Lemma \ref{lem:rosefree}. 
\end{proof}

\begin{ex}
The braid group $B_n(K_m)$ is hyperbolic relative to abelian subgroups if and only if $n=1$, or $n=2$ and $m \leq 7$, or $n=3$ and $m \leq 4$, or $m \leq 3$. The non hyperbolic groups in this family are $B_2(K_7)$ and $B_3(K_6)$. 
\end{ex}

\begin{ex}
The braid group $B_n(K_{p,q})$ is hyperbolic relative to abelian subgroups if and only if $n=1$, or $n=2$ and $p,q \leq 4$, or $n=3$ and $p,q \leq 3$, or $n=4$ and $p=2$ and $q \leq 3$, or $p,q \leq 2$. The non hyperbolic groups in this family are $B_2(K_{4,4})$, $B_3(K_{3,3})$ and $B_4(K_{2,3})$. 
\end{ex}

\noindent
Despite the existence of a general criterion determining whether the fundamental group of a compact special cube complex is relatively hyperbolic \cite{SpecialRH}, we were not able to describe when a graph braid group is relatively hyperbolic. This question remains open in full generality; see Question \ref{q:BraidRH}. Nevertheless, we are able to state and prove directly a sufficient criterion of relative hyperbolicity. Before stating our criterion, we need to introduce the following definition:

\begin{definition}
Let $\Gamma$ be a one-dimensional CW-complex, $n \geq 1$ an integer, $\Lambda \subset \Gamma$ a subgraph and $w \in \mathcal{D}(UC_n(\Gamma))$ a diagram. The \emph{coset} $w \langle \Lambda \rangle$ is the set of diagrams which can be written as a concatenation $w \cdot \ell$ where $\ell$ is a diagram represented by a legal word of oriented edges belonging to $\Lambda$. Similarly, for every diagrams $a,b \in \mathcal{D}(UC_n(\Gamma))$, we denote by $a \langle \Lambda \rangle b$ the set of diagrams which can be written as a concatenation $a \cdot \ell \cdot b$ where $\ell$ is a diagram represented by a legal word of oriented edges belonging to $\Lambda$. 
\end{definition}

\noindent
Our main criterion is the following:

\begin{thm}\label{thm:braidgrouprh}
Let $\Gamma$ be a connected compact one-dimensional CW-complex, and let $\mathcal{G}$ be a collection of subgraphs of $\Gamma$ satisfying the following conditions:
\begin{itemize}
	\item every pair of disjoint simple cycles of $\Gamma$ is contained in some $\Lambda \in \mathcal{G}$;
	\item for every distinct $\Lambda_1,\Lambda_2 \in \mathcal{G}$, the intersection $\Lambda_1 \cap \Lambda_2$ is either empty or a disjoint union of segments;
	\item if $\gamma$ is a reduced path between two vertices of some $\Lambda \in \mathcal{G}$ which is disjoint from some cycle of $\Lambda$, then $\gamma \subset \Lambda$.
\end{itemize}
Then $B_2(\Gamma)$ is hyperbolic relative to subgroups which are isomorphic to $w \langle \Lambda \rangle w^{-1}$ for some subgraph $\Lambda \in \mathcal{G}$ and some diagram $w$ satisfying $t(w) \subset \Lambda$. In particular, $B_2(\Gamma)$ is relatively hyperbolic if $\mathcal{G}$ is a collection of proper subgraphs. 
\end{thm}

\begin{proof}
Set $\mathcal{C}= \{ w \langle \Lambda \rangle \mid \Lambda \in \mathcal{G}, \ t(w) \subset \Lambda \}$. We think of $\mathcal{C}$ as a collection of convex subcomplexes of the CAT(0) cube complex $X_n(\Gamma)$. For convenience, set $L = \# V(\Gamma)(\#V(\Gamma)-1)/2$ which is also the number of vertices of $UC_2(\Gamma)$. Recall that a \emph{flat rectangle} in $X_n(\Gamma)$ is an isometric combinatorial embedding $[0,p] \times [0,q] \hookrightarrow X_n(\Gamma)$; it is \emph{$K$-thick} for some $K \geq 0$ if $p,q > K$. 

\begin{claim}\label{claim:rh1}
For every distinct $C_1,C_2 \in \mathcal{C}$, there exist at most $L$ hyperplanes intersecting both $C_1$ and $C_2$.
\end{claim}

\noindent
Let $w_1 \langle \Lambda_1 \rangle, w_2 \langle \Lambda_2 \rangle \in \mathcal{C}$ be two cosets both intersected by at least $L+1$ hyperplanes. As a consequence of \cite[Corollary 2.17]{coningoff}, there exists a flat rectangle $[0,p] \times [0,q] \hookrightarrow X_n(\Gamma)$ such that $\{0\} \times [0,q] \subset w_1 \langle \Lambda_1 \rangle$, $\{ p \} \times [0,q] \subset w_2 \langle \Lambda_2 \rangle$ and $q \geq L+1$. Let $\Lambda_0$ be the union of the edges of $\Gamma$ labelling the hyperplanes intersecting the path $\{0\} \times [0,q]$. Since the edges of $w_1 \langle \Lambda_1 \rangle$ (resp. $w_2 \langle \Lambda_2 \rangle$) are all labelled by edges of $\Lambda_1$ (resp. $\Lambda_2$), it follows that $\Lambda_0 \subset \Lambda_1 \cap \Lambda_2$. Reading the word labelling the path $\{0\} \times [0,q]$, we deduce that there exists a reduced diagram of length at least $L+1$ in $\langle \Lambda_0 \rangle$; a fortiori, by looking at the prefixes to the previous word, we know that there exists an initial configuration $S \in UC_n(\Gamma)$ (which is the terminus of $(0,0)$) for which there exist at least $L+1$ distinct $(S, \ast)$-diagrams in $\langle \Lambda \rangle$. Necessarily, at least two such diagrams, say $A$ and $B$, must have the same terminus since there are at most $L$ possible terminus. Then $AB^{-1}$ provides a non-trivial spherical diagram in $\langle \Lambda \rangle$, showing that $B_2(\Lambda_0,S)$ is non-trivial, and finally that $\Lambda_0$ is not a union of segments. This implies that $\Lambda_1$ and $\Lambda_2$ must be equal; for convenience, let $\Lambda$ denote this common subgraph of $\Gamma$.

\medskip \noindent
Let $m = m_1 \cdots m_p$ denote the word labelling the path $[0,p] \times \{ 0 \}$. We claim that $m_1 \cup \cdots \cup m_p$ defines a path in $\Gamma$; in other words, thought of as a motion of two particles starting from the initial configuration $S$ which is the terminus of $(0,0)$, $m$ fixes a particle. Indeed, if $e$ denotes the edge of $\Gamma$ labelling $\{0\} \times [0,1]$ (which exists since $q \geq 1$), then the $m_i$'s are disjoint from $e$ since the hyperplane dual to $e$ is transverse to any hyperplane intersecting $[0,p] \times \{ 0\}$, so that $m$ must fixes the starting point of $e$. This proves that $m$ is a path, and even a reduced path since $[0,p] \times \{ 0 \}$ is a geodesic in $X_n(\Gamma)$. In fact, our argument implies the following more general statement:

\begin{fact}\label{fact:pathfromword}
Let $[0,p] \times [0,q] \hookrightarrow X_2(\Gamma)$ be a flat rectangle. For every $0 \leq k \leq q$, the word of edges of $\Gamma$ labelling the path $[0,p] \times \{ k \}$ defines a regular path in $\Gamma$ when thought of as a sequence of edges of $\Gamma$. 
\end{fact}

\noindent
As a consequence, if $p \geq 1$, then the word of edges of $\Lambda_0$ labelling $\{0\} \times [0,q]$ defines a reduced path in $\Gamma$. But this path has length $q \geq L+1$ in a subgraph $\Lambda_0$ containing at most $L$ vertices. Consequently, $\Lambda_0$ must contain an induced cycle. 

\medskip \noindent
Let $w_1 \ell_1$ and $w_2 \ell_2$ denote the vertices $(0,0)$ and $(p,0)$ respectively, where $\ell_1 \in \langle \Lambda_1 \rangle$ and $\ell_2 \in \langle \Lambda_2 \rangle$. Notice that the starting point of $m$ belongs to the terminus of $w_1 \ell_1$, which is included in $\Lambda$; and that the ending point of $m$ belongs to the terminus of $w_1 \ell_1 p = w_2 \ell_2$, which is included in $\Lambda$. Therefore, $m$ is a reduced path in $\Gamma$ between two vertices of $\Lambda$, which is disjoint from $\Lambda_0$ since any hyperplane intersecting $[0,p] \times \{ 0 \}$ is transverse to any hyperplane intersecting $\{0 \} \times [0,q]$. 

\medskip \noindent
Two cases may happen, either $p \geq 1$, so that $\Lambda_0 \subset \Lambda$ contains an induced cycle disjoint from $p$; or $p=0$, so that $w_1 \langle \Lambda \rangle$ and $w_2 \langle \Lambda \rangle$ intersect. In the former case, it follows from our assumptions that $m_1, \ldots, m_p \in \Lambda$, so that $[0,p] \times [0,q] \subset w_1 \langle \Lambda \rangle$. A fortiori, $w_1 \langle \Lambda \rangle$ and $w_2 \langle \Lambda \rangle$ must intersect as well. Since two distinct cosets associated to the same subgraph of $\Gamma$ must be disjoint, we deduce that $w_1 \langle \Lambda \rangle = w_2 \langle \Lambda \rangle$, concluding the proof of our claim.

\begin{claim}\label{claim:rh2}
Every $2L$-thick flat rectangle of $X_2(\Gamma)$ is contained in the $2L$-neighborhood of some element of $\mathcal{C}$.  
\end{claim}

\noindent
Let $[0,p] \times [0,q] \hookrightarrow X_2(\Gamma)$ be an $2L$-thick flat rectangle. Let $0 \leq a<b \leq L+1$ and $p-L-1 \leq c < d \leq p$ be four integers such that $(a,0)$ and $(b,0)$ (resp. $(c,0)$ and $(d,0)$) have the same terminus, and such that two distinct vertices of $(a,b) \times \{0\}$ (resp. $(c,d) \times \{ 0\}$) do not have the same terminus. Similarly, let $0 \leq e < f \leq L+1$ and $q-L-1 \leq g<h \leq q$ be four integers such that $(0,e)$ and $(0,f)$ (resp. $(0,g)$ and $(0,h)$) have the same terminus, and such that two distinct vertices of $\{0\} \times (e,f)$ (resp. $\{0 \} \times (g,h)$) do not have the same terminus. In order to deduce our claim, it is sufficient to show that there exists some $C \in \mathcal{C}$ containing $[a,d] \times [e,h]$. 

\medskip \noindent
Let $y$ be the word of edges of $\Gamma$ labelling the path $[a,d] \times \{ e \}$. According to Fact \ref{fact:pathfromword}, thought of as a sequence of edges of $\Gamma$, $y$ defines a reduced path in $\Gamma$; let $\alpha$ denote this path. Notice that, for every $1 \leq x<y \leq p$, the vertices $(x,e)$ and $(y,e)$ have the same terminus if and only if $(x,0)$ and $(y,0)$ have the same terminus themselves since $(x,e)=(x,0) \cdot E$ and $(y,e)=(y,0) \cdot E$ where $E$ denotes the diagram labelling the path $\{0 \} \times [0,e]$ (which also labels $\{x \} \times [0,e]$ and $\{y \} \times [0,e]$). A fortiori, $(a,e)$ and $(b,e)$ (resp. $(c,e)$ and $(d,e)$) have the same terminus, and two distinct vertices of $(a,b) \times \{e\}$ (resp. $(c,d) \times \{ e\}$) do not have the same terminus. Therefore the subpaths of $\alpha$ corresponding to $[a,b] \times \{e \}$ and $[c,d] \times \{ e \}$ are simple cycles of $\Gamma$. In other words, if $\Lambda(y) \subset \Gamma)$ denotes the union of the edges of $\alpha$, then $\Lambda(y)$ decomposes as two simple cycles linked by a reduced path. Similarly, if $z$ denotes the word of edges of $\Gamma$ labelling the path $\{ a \} \times [e,h]$, then, thought of as a sequence of edges of $\Gamma$, it defines a reduced path $\beta$ in $\Gamma$ such that, if $\Lambda(z)$ denotes the union of the edges of $\beta$, then $\Lambda(z) \subset \Gamma$ decomposes as two simple cycles linked by a reduced path. Notice that $[a,d] \times [e,h] \subset w \langle \Lambda(y) \cup \Lambda(z) \rangle$, where $w=(a,e)$; and that $\Lambda(y) \cap \Lambda(z)= \emptyset$, since any hyperplane intersecting $[a,d] \times \{e \}$ must be transverse to any hyperplane intersecting $\{a\} \times [e,h]$. 

\medskip \noindent
Let convenience, let $A,B$ (resp. $C,D$) denote our two simple cycles in $\Lambda(y)$ (resp. in $\Lambda(z)$). By assumption, there exist some $\Lambda_1,\Lambda_2,\Lambda_3$ such that $A \cup C \subset \Lambda_1$, $A \cup D \subset \Lambda_2$ and $B \cup D \subset \Lambda_3$. By noticing that the intersections $\Lambda_1 \cap \Lambda_2$ and $\Lambda_2 \cap \Lambda_3$ are not unions of segments (since they contain cycles), it follows that $\Lambda_1= \Lambda_2= \Lambda_2$; let $\Lambda$ denote this common subgraph. So $\Lambda$ belongs to $\mathcal{G}$ and contains $A \cup B \cup C \cup D$. Moreover, the reduced path between $A$ and $B$ in $\Lambda(y)$ links two vertices of $\Lambda$ and is disjoint from $D$ and $C$ (which are included in $\Lambda(z)$), so this path must be included in $\Lambda$. Similarly, the path between $C$ and $D$ in $\Lambda(y)$ must be included in $\Lambda$. A fortiori, $\Lambda(y)$ and $\Lambda(z)$ are both contained in $\Lambda$. Consequently,
$$[a,d] \times [e,h] \subset w \langle \Lambda(y) \cup \Lambda(z) \rangle \subset w \langle \Lambda \rangle \in \mathcal{C},$$
which concludes the proof of our claim.

\begin{claim}\label{claim:rh3}
An edge of $X_2(\Gamma)$ belongs to only finitely many subcomplexes of $\mathcal{C}$.
\end{claim}

\noindent
If $w_1 \langle \Lambda_1 \rangle, w_2 \langle \Lambda_2 \rangle \in \mathcal{C}$ are distinct and both contain a given edge of $X_2(\Gamma)$, necessarily $\Lambda_1 \neq \Lambda_2$. Therefore, at most $\# \mathcal{G} \leq 2^{\# V(\Gamma)}$ subcomplexes of $\mathcal{C}$ contain a given edge of $X_2(\Gamma)$, proving our claim. 

\medskip \noindent
Now, we are ready to conclude the proof of our theorem. Let $Y$ denote the cone-off of $X_2(\Gamma)$ over $\mathcal{C}$. More precisely, $Y$ is the graph obtained from the one-skeleton of $X_2(\Gamma)$ by adding a vertex for each $C\in \mathcal{C}$ and by adding an edge between this vertex and any vertex of $C$. According to \cite[Theorem 4.1]{coningoff}, it follows from Claim \ref{claim:rh2} that $Y$ is hyperbolic; and according to \cite[Theorem 5.7]{coningoff}, it follows from Claims \ref{claim:rh1} and~\ref{claim:rh3} that $Y$ is fine. In only remaining point to show in order to deduce that $B_2(\Gamma)$ is hyperbolic relative to 
$$\{ w \langle \Lambda \rangle w^{-1} \mid \Lambda \in \mathcal{G}, \ t(w) \subset \Lambda \}$$
is that $Y$ contains finitely many orbits of edges. In fact, it is sufficient to prove that, for every $\Lambda \in \mathcal{G}$ and every diagram $w$, the subgroup $w \langle \Lambda \rangle w^{-1}$ acts cocompactly on $w \langle \Lambda \rangle$. For that purpose, just notice that if $w \ell_1, w \ell_2 \in w \langle \Lambda \rangle$ have the same terminus then $w \ell_1 \ell_2^{-1} w^{-1}$ is an element of $w \langle \Lambda \rangle w^{-1}$ sending $w \ell_2$ to $w \ell_1$. 
\end{proof}

\begin{ex}
Let $\Gamma$ be a union of two bouquets of circles (each containing at least one circle) whose centers are linked by a segment (which is not reduced to a single vertex). Notice that $\Gamma$ contains a pair of disjoint induced cycles, so that $B_2(\Gamma)$ is not hyperbolic according to Theorem \ref{thm:hypbraidgroup}. Nevertheless, if $\Lambda \subset \Gamma$ denotes the union of the two bouquets of circles of $\Gamma$, then Theorem \ref{thm:braidgrouprh} applies to $\mathcal{G}= \{ \Lambda \}$, so that $B_2(\Gamma)$ is (non-trivially) hyperbolic relative to subgroups which are isomorphic to $B_2(\Lambda,S)$ for some configuration $S \in UC_2(\Lambda)$. Notice that $B_2(\Lambda_2,S)$ is isomorphic to the product of free groups $\mathbb{F}_r \times \mathbb{F}_s$ if $S$ contains a particle in each connected component of $\Lambda$, where $r$ and $s$ denotes the number of circles contained in each bouquets; and $B_2(\Lambda,S)$ is free if $S$ is contained in a single connected component of $\Lambda$ (as a consequence of Remark~\ref{remark:criterionfree}). Consequently, the braid group $B_2(\Gamma)$ is hyperbolic relative to a finite collection of groups isomorphic to $\mathbb{F}_r \times \mathbb{F}_s$. 
\end{ex}

\begin{ex}
Let $\Gamma= K_2^{\mathrm{opp}} \ast C_4$ be the graph obtained from a square by adding two new vertices and by linking them by an edge to any vertex of the square. Let $\mathcal{G}$ denote the collection of pairs of disjoint triangles of $\Gamma$. Then Theorem \ref{thm:braidgrouprh} applies, showing that $B_2(\Gamma)$ is hyperbolic relative to subgroups which are isomorphic to $B_2(\Lambda,S)$ for some $UC_2(\Lambda)$, where $\Lambda$ is a disjoint union of two triangles. Notice that either $S$ is contained in a single triangle of $\Lambda$, so that $B_2(\Lambda,S) \simeq \mathbb{Z}$; or $S$ intersects the two connected components of $\Lambda$, so that $B_2(\Lambda,S) \simeq \mathbb{Z}^2$. Consequently, the braid group $B_2(\Gamma)$ is hyperbolic relative to free abelian subgroups of rank two.
\end{ex}

\begin{ex}
Consider the complete graph on six vertices $K_6$. Let $\mathcal{G}$ denote the collection of pairs of disjoint triangles of $K_6$. Then Theorem \ref{thm:braidgrouprh} applies, showing, as in the previous examples, that the braid group $B_2(K_6)$ is hyperbolic relative to free abelian subgroups of rank two.

\medskip \noindent 
It is worth noticing that Theorem \ref{thm:braidgrouprh} does not apply to $B_2(K_n)$ for $n \geq 7$. We do not know whether or not these groups are relatively hyperbolic.
\end{ex}

\begin{ex}
Consider the bipartite complete graph $K_{4,4}$. Let $\mathcal{G}$ denote the collection of pairs of disjoint squares of $K_{4,4}$. Then Theorem \ref{thm:braidgrouprh} applies, showing that the braid group $B_2(K_{4,4})$ is hyperbolic relative to free abelian subgroups of rank two.

\medskip \noindent
It is worth noticing that Theorem \ref{thm:braidgrouprh} does not apply to $B_2(K_{m,n})$ for $m \geq 5$ or $n \geq 5$. We do not know whether or not these groups are relatively hyperbolic.
\end{ex}

\section{Open questions}\label{section:questions}

\noindent
We conclude this paper by setting several open questions. The first one is naturally inspired by Proposition \ref{prop:braidembed}.

\begin{question}\label{q:embedding}
Let $\Gamma$ be a connected one-dimensional CW-complex and $n \geq 2$ an integer. Is it true that $B_m(\Gamma)$ embeds into $B_n(\Gamma)$ for every $m \leq n$?
\end{question}

\noindent
Theorem \ref{thm:hypbraidgroup} determines which graph braid groups are hyperbolic. Conversely, an interesting question is: 

\begin{problem}\label{problem:hypbraidgroups}
Which hyperbolic groups arise as graph braid groups?
\end{problem}

\noindent
We know from Theorem \ref{thm:hypbraidgroup} that, when $n \geq 4$, there are only free groups. When $n=3$, it remains only two cases to study:

\begin{question}
Are the graph braid groups of sun graphs and pulsar graphs free? one-ended? surface groups? 3-manifold groups?
\end{question}

\noindent
For $n =2$, the situation is less clear. Notice that not all graph braid groups on two particles are free since we gave one-ended hyperbolic graph braid groups in Examples~\ref{ex1:GraphHyp} and \ref{ex2:GraphHyp}. Figure \ref{figure4} gives examples of graphs which do not satisfy the criterion given in Remark \ref{remark:criterionfree}; we do not know whether or not the corresponding braid groups are free. Nevertheless, many hyperbolic graph braid groups turn out to be free, stressing out the following question:
\begin{figure}
\begin{center}
\includegraphics[trim={0 20.5cm 19cm 0},clip,scale=0.44]{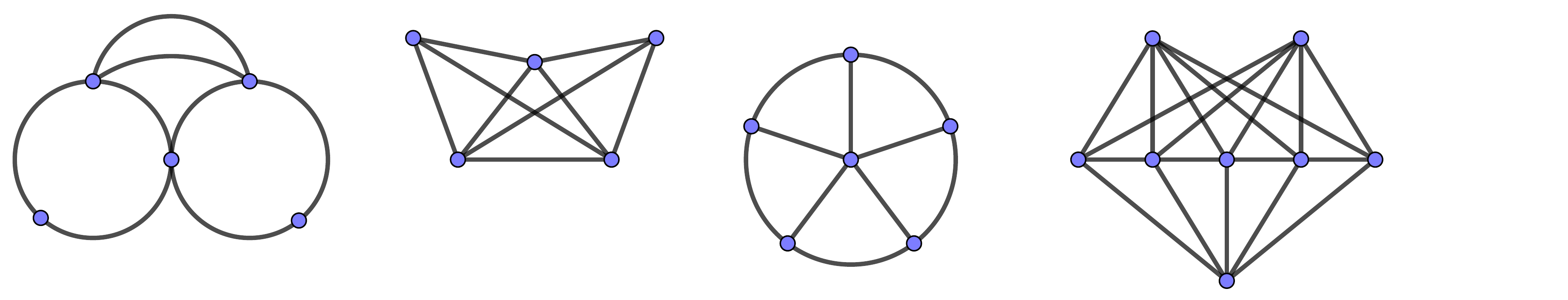}
\caption{If $\Gamma$ is one of these graphs, is $B_2(\Gamma)$ free?}
\label{figure4}
\end{center}
\end{figure}

\begin{question}
When is a graph braid group free?
\end{question}

\noindent
It is worth noticing that, if $\Gamma$ is a connected non-planar graph, then the braid group $B_2(\Gamma)$ is not free, since we know from \cite{KimKoPark} that its abelianisation contains finite-order elements. Consequently, we know that there exist many graphs leading to non-free hyperbolic braid groups. A possibly easier question, but interesting as well, is:

\begin{question}
When is a graph braid group one-ended?
\end{question}

\noindent
As a consequence of Stallings' theorem, an answer to this question would provide a characterisation of graph braid groups splitting as free products. 

\medskip \noindent
It is worth noticing that, in the third point of Theorem \ref{thm:braidRHabelian}, there are only finitely many graphs $\Gamma$ whose braid groups on four particles could not be free. Therefore, it would interesting to determine these groups.

\begin{question}
What is the braid group $B_4(\Gamma)$ if $\Gamma$ is a segment linking two vertices of degree three, or a cycle containing two vertices of degree three, or two cycles glued along a non-trivial segment? Are these groups free products?
\end{question}

\noindent
Finally, as mentioned in the previous section, we were not able to determine precisely when a graph braid group is relatively hyperbolic. So the question remains open in full generality.

\begin{question}\label{q:BraidRH}
When is a graph braid group (on two particles) relatively hyperbolic?
\end{question}

\noindent
For instance, a natural question to ask is: does the converse of Theorem \ref{thm:braidgrouprh} hold?

\addcontentsline{toc}{section}{References}

\bibliographystyle{alpha}
{\footnotesize\bibliography{SpecialAndGraphBraidGroup}}

\end{document}